\documentclass[12pt]{article}

\usepackage[inner=32mm,outer=31mm,tmargin=30mm,bmargin=40mm]{geometry}
\usepackage[ngerman,english]{babel}

\usepackage{latexsym,amsfonts,amsmath,amssymb,epsfig,tabularx,amsthm,dsfont,mathrsfs}

\usepackage{graphicx}
\usepackage{enumerate}

\usepackage{booktabs}
\usepackage{titling}
\newcommand{\subtitle}[1]{%
  \posttitle{%
    \par\end{center}
    \begin{center}\large#1\end{center}
    \vskip0.5em}%
}

\usepackage{hyperref} 
\hypersetup{colorlinks=true,
        linkcolor=black,
        citecolor=black,
        urlcolor=blue}

\usepackage{pgfplots}

\usepackage{framed}
\usepackage{amscd}

\usepackage{tikz-cd}
\usepackage{tikz}

\theoremstyle{plain}

\newtheorem{theorem}{Theorem}[section]
\newtheorem{lemma}[theorem]{Lemma}

\newtheorem{proposition}[theorem]{Proposition}
\newtheorem{corollary}[theorem]{Corollary}
\newtheorem{definition}[theorem]{Definition}

\theoremstyle{definition}
\newtheorem{remark}[theorem]{Remark}
\newtheorem{example}[theorem]{Example}

\renewcommand{\P}{{\mathbb P}}
\newcommand{\expect}{\operatorname{\mathbb{E}}}

\DeclareMathOperator{\Uniform}{unif}

\newcommand{\iid}{\stackrel{\textup{iid}}{\sim}}

\DeclareMathOperator{\EquiHash}{EquiHash}

\DeclareMathOperator{\Sym}{Sym}

\newcommand{\ind}{\mathds{1}}
\DeclareMathOperator{\sgn}{sgn}

\DeclareMathOperator{\median}{med}

\DeclareMathOperator{\card}{card}
\newcommand{\ran}{\textup{ran}}
\newcommand{\deter}{\textup{det}}

\newcommand{\nonada}{\textup{nonada}}

\newcommand{\eps}{\varepsilon}
\newcommand{\embed}{\hookrightarrow}

\renewcommand{\vec}{\boldsymbol}

\newcommand{\R}{{\mathbb R}}

\newcommand{\N}{{\mathbb N}}

\DeclareMathAlphabet{\mathup}{OT1}{\familydefault}{m}{n}

\DeclareMathOperator{\shrink}{Shrink}
\DeclareMathOperator{\precond}{Precond}
\DeclareMathOperator{\spot}{Spot}

\DeclareMathOperator{\discover}{Discover} 
\DeclareMathOperator{\linsketch}{LinSketch}
\DeclareMathOperator{\countsketch}{CountSketch}

\newcommand{\widebar}[1]{\mbox{\kern1.5pt\hbox{\vbox{\hrule height 0.6pt \kern0.35ex
        \hbox{\kern-0.15em \ensuremath{#1 }\kern0.0em}}}}\kern-0.1pt}

\newlength{\fixboxwidth}
\usepackage{soul}

\usepackage{latexsym,amsfonts,amsmath,amssymb,mathrsfs}
\usepackage{url}
\usepackage{graphicx}
\usepackage{color}
\usepackage{euscript}
\usepackage{verbatim}
\usepackage{hyperref}

\definecolor{darkgreen}{rgb}{0,0.5,0}

\begin{document}

\title{Adaptive and non-adaptive randomized approximation of high-dimensional vectors}

\author{Robert J. Kunsch\thanks{RWTH Aachen University,
Analysis and its Applications,
Pontdriesch~10,
52062 Aachen, Email: kunsch@mathc.rwth-aachen.de},
Marcin Wnuk\thanks{Institut für Mathematik, 
Osnabrück University, Albrechtstraße 28a, 49076 Osnabrück, 
Email: marcin.wnuk@uni-osnabrueck.de}}

\date{\today}

\maketitle
\begin{abstract}
We study approximation of the embedding $\ell_p^m \embed \ell_q^m$, $1 \leq p < q \leq \infty$,
based on randomized algorithms that use up to $n$ arbitrary linear 
functionals as information on a problem instance where $n \ll m$.
By analysing adaptive methods we
show upper bounds
for which the information-based complexity~$n$ 
exhibits only a $(\log\log m)$-dependence.
In the case $q < \infty$ we use a multi-sensitivity approach in order to reach optimal polynomial order in $n$ for the Monte Carlo error.
We also improve on non-adaptive methods for $q < \infty$ by denoising 
known algorithms for uniform approximation.
\end{abstract}

{\bf Keywords: } Monte Carlo, information-based complexity, upper bounds,
adaption, confidence

\section{Introduction}

We continue our study from ~\cite{KNW24,KW24b} on the $\ell_q$-approximation of vectors \mbox{$\vec{x} \in \R^m$} relative to their $\ell_p$-norm
via algorithms $A_n$ that use at most $n$~adaptively chosen randomized linear measurements
of $\vec{x}$.
We consider $1 \leq p < q \leq \infty$ and write $\ell_p^m \embed \ell_q^m$ as a short-hand for the problem.
Studying these sequence space embeddings
is foundational for understanding the approximation of many other more complicated linear problems
such as function approximation, see for instance~\cite{H92,Ma91} in the context of randomized approximation.

A randomized \emph{algorithm} (\emph{method}, \emph{scheme}) $A = (A^\omega)_{\omega \in \Omega}$ 
for the approximation problem $\ell_p^m \embed \ell_q^m$
is a family of mappings $A^\omega : \R^m \to \R^m$
indexed by elements of an underlying probability space~$(\Omega,\Sigma,\P)$,
for any fixed~$\vec{x} \in \R^m$ the output $A(\vec{x})\colon \omega \mapsto A^\omega(\vec{x})$ is a random variable. 
We restrict to mappings where the output is generated
based on finitely many pieces of information $y_i = L_i(\vec{x}) \in \R$ with linear functionals~$L_i$ that may depend on $\omega \in \Omega$, thus being random.
In addition we allow for \emph{adaptivity}, that is, the choice of $L_i$ (or the decision whether we even want to continue collecting information) may also depend on previously obtained information $y_1,\ldots,y_{i-1}$.
Writing $k = k(\omega,\vec{x}) \in \N_0$ for the number of evaluated information functionals, 
the output of the algorithm is of the form $A(\vec{x}) = \phi(y_1,\ldots,y_k)$,
where the reconstruction map $\phi$ may also depend on $\omega \in \Omega$.
We consider algorithms with a strict bound on the number of pieces of information,
the \emph{cardinality} of~$A$ (also called \emph{information cost}) is defined as
\begin{equation*}
    \card A := \sup_{\omega,\vec{x}} k(\omega,\vec{x}) \,.
\end{equation*}
The error of $A$ is defined by the worst case expectation
\begin{equation} \label{eq:e(A_n)}
    e(A,\ell_p^m \embed \ell_q^m)
        := \sup_{\|\vec{x}\|_p \leq 1} \expect \|\vec{x} - A(\vec{x})\|_q
\end{equation}
with the classical $\ell_p$-norms ($1 \leq p < \infty$)
\begin{equation*}
    \|\vec{x}\|_p := \left(|x_1|^p + \ldots + |x_m|^p\right)^{1/p},
    \quad \text{and} \quad \|\vec{x}\|_\infty := \max_{i=1,\ldots,m} |x_i| \,.
\end{equation*}
We aim for the error of optimal randomized algorithms, hence, the quantity of interest is
\begin{equation} \label{eq:e(n)}
    e^{\ran}(n,\ell_p^m \embed \ell_q^m)
        := \inf_{A_n} e(A_n,\ell_p^m \embed \ell_q^m) \,,
\end{equation}
where the infimum is taken over algorithms $A_n$ with cardinality at most~$n$.
Conversely, for $\eps > 0$ we define the (Monte Carlo) $\eps$-complexity of a problem by
\begin{equation*}
    n^{\ran}(\eps,\ell_p^m \embed \ell_q^m)
        := \inf\left\{\card A \;\middle|\;
                        \text{algorithms $A$ with }
                        e(A,\ell_p^m \embed \ell_q^m) \leq \eps
                \right\} .
\end{equation*}

The key finding of~\cite{KNW24,KW24b} was that under certain circumstances
the minimal error~\eqref{eq:e(n)} can only be achieved with \emph{adaptive} methods.
In other words, the error will be considerably larger if we restrict to \emph{non-adaptive} methods
where 
all information functionals $L_i$ are chosen independently of the input~$\vec{x}$ and
the information can be written as $\vec{y} = N^\omega \vec{x} \in \R^n$ with a random matrix \mbox{$N^\omega \in \R^{n \times m}$}.
Analogously to~\eqref{eq:e(n)},
we define $e^{\ran,\nonada}(n,\ell_p^m \embed \ell_q^m)$
and $n^{\ran,\nonada}(\eps,\ell_p^m \embed \ell_q^m)$
with the infimum taken over non-adaptive methods.

In~\cite[Thm~2.7]{KNW24} a lower bound for non-adaptive methods was shown, namely,
with suitable constants~$C,a > 0$,
\begin{equation} \label{eq:nonadaLB}
    e^{\ran,\nonada}(n,\ell_p^m \embed \ell_q^m)
        \geq e^{\ran,\nonada}(n,\ell_1^m \embed \ell_\infty^m)
        \geq \frac{1}{100}
    \qquad \text{for } m \geq C \cdot e^{a n^2}.
\end{equation}
In other words, for large problem size~$m$ relative to the cardinality~$n$
(or for small cardinality $n \leq c \sqrt{\log m}$ compared to the dimension~$m$, with suitable $c > 0$)
the error of non-adaptive methods
cannot be significantly smaller than the initial error
\mbox{$e^{\ran}(0,\ell_p^m \embed \ell_q^m) = 1$}.
The follow-up paper~\cite{KW24b} was dedicated to adaptive algorithms for uniform approximation ($q=\infty$),
showing
\begin{equation} \label{eq:uniformUB}
    e^{\ran}(n,\ell_p^m \embed \ell_\infty^m)
        \preceq \min\left\{1,\, \left(\frac{\log n + \log \log m}{n}\right)^{\frac{1}{p}} \right\} 
    \qquad\text{for } 1 \leq p \leq 2 \,,
\end{equation}
see~\cite[Thm~3.3]{KW24b}. Together with the lower bounds~\eqref{eq:nonadaLB},
this lead to a gap of order~$n$ (up to logarithmic terms) between adaptive and non-adaptive approximation~\cite[Thm~4.1]{KW24b}:
\begin{equation} \label{eq:gap1oo}
    \frac{e^{\ran}(n,\ell_1^m \embed \ell_\infty^m)
        }{e^{\ran,\nonada}(n,\ell_1^m \embed \ell_\infty^m)}
        \preceq \frac{\log n}{n}
    \qquad\text{for $m = m_n := \left\lceil C\, e^{an^2} \right\rceil$.}
\end{equation}
The algorithm presented in~\cite{KW24b} is based on ideas of Woodruff and different co-authors~\cite{IPW11,LNW17,LNW18} who studied related problems of \emph{stable sparse recovery}.
Already in~\cite[Thm~3.1]{KNW24} we cited results from~\cite{IPW11,LNW17,LNW18} to show
\begin{equation} \label{eq:ell2UB}
    e^{\ran}(n,\ell_1^m \embed \ell_2^m)
        \preceq \min\left\{1,\, \sqrt{\frac{\log \log \frac{m}{n}}{n}} \right\} ,
\end{equation}
leading to a gap of order~$\sqrt{n}$ up to logarithmic terms~\cite[Cor~3.3]{KNW24}:
\begin{equation} \label{eq:gap12}
    \frac{e^{\ran}(n,\ell_1^m \embed \ell_2^m)
        }{e^{\ran,\nonada}(n,\ell_1^m \embed \ell_2^m)}
        \preceq \sqrt{\frac{\log n}{n}}
    \qquad\text{for $m = m_n := \left\lceil C\, e^{an^2} \right\rceil$.}
\end{equation}
In fact, the gaps~\eqref{eq:gap1oo} and \eqref{eq:gap12} are optimal up to logarithmic terms, see~\cite{KNU24}.
In this paper we continue the presentation of~\cite{KW24b} to give a precise description of an algorithm for the problem~$\ell_p^m \embed \ell_q^m$ in the regime
$q < \infty$,
in particular, for $1 \leq p \leq 2$ and $p < q < \infty$ we show
\begin{equation*}
    e^{\ran}(n,\ell_p^m \embed \ell_q^m)
        \preceq \min\left\{1,\, \left(\frac{\log \log \frac{m}{n}}{n}\right)^{\frac{1}{p} - \frac{1}{q}} \right\}
\end{equation*}
which contains the bound~\eqref{eq:ell2UB} as a special case,
see Theorem~\ref{thm:complexity}.
This algorithm directly designed for our problem is in fact simpler than the algorithm for stable sparse recovery by Woodruff et al.~\cite{LNW17,LNW18} on which we relied in~\cite{KNW24}.
By this we obtain gaps between adaptive and non-adaptive methods for a broader range of regimes, see Theorem~\ref{thm:adagap3}.

The phenomenon that adaptive randomized algorithms can be superior to non-adaptive randomized algorithms for some linear problems was first demonstrated by Heinrich~\cite{He24,He22,He23b,He23}. 
Here again, finite-dimensional problems~\cite{He24,He22,He23}
provide a starting point for the study of problems 
in function spaces~\cite{He23b}. 
Those results were quite surprising since for linear problems in the deterministic setting non-adaptive algorithms are almost as good as adaptive ones, see the survey~\cite{No96}.

The paper is structured as follows: In Section~\ref{sec:tools} we start by reviewing parts of the algorithm for uniform approximation from our previous paper~\cite{KW24b} and introduce slight modifications.
We combine them in Section~\ref{sec:discover} to give a higher-level routine $\discover$
which is designed to detect entries of a vector with adjustable sensitivity.
Section~\ref{sec:precond} discusses a minor improvement of this routine
without affecting the weak asymptotic analysis.
In Section~\ref{sec:Approx} the full approximation scheme which combines several instances of $\discover$ with varying sensitivities is presented.
From this we draw conclusions on the complexity of the approximation problem.
Finally, to complete the picture, in Section~\ref{sec:nonada} we improve on upper bounds for non-adaptive methods in the regime of $n \ll m$.
In particular, we establish a denoising methodology for algorithms performing well in uniform approximation,
see Lemma~\ref{lem:denoising}.
Results on non-adaptive methods are summarized in Theorem~\ref{thm:nonada-summary}.

\subsection*{Asymptotic notation}

We use asymptotic notation to compare functions $f$ and $g$ that depend on variables $(\eps,m)$ or $(n,m)$, writing \mbox{$f \preceq g$} if there exists a constant \mbox{$C > 0$} such that $f \leq C g$ holds for ``small~$\eps$ and large values of $m \eps^p$''
(say, for \mbox{$0 < \eps < \frac{1}{2}$} and~$m\eps^p \geq 16$), 
or for ``$n \ll m$'' (say, $\frac{m}{n} \geq 16$), respectively.
\emph{Weak asymptotic equivalence} $f \asymp g$ means $f \preceq g \preceq f$.
We also use \emph{strong asymptotic equivalence} $f \simeq g$ to state $f/g \to 1$ for $\eps \to 0$ and $m\eps^p \to \infty$, or $n \to \infty$ and $\frac{m}{n} \to \infty$, respectively.
The implicit constant~$C$ or the convergence $f/g$ is to be understood for fixed values of the parameters $p$ and $q$.

A recurring theme in this paper are bounds that contain a double logarithm $\log \log x$ as a factor.
This is defined for $x > 1$ and monotonically increasing, but it only exhibits positive values for $x > e$.
However, we also need that such a factor is greater than a positive constant,
thus we restrict ourselves to $x \geq 16$ which ensures $\log \log x > 1$,
at the price of constraints on the domain of asymptotic relations.
Asymptotic relations under varying restrictions can be challenging,
so in Appendix~\ref{app:asymp} we prove a
less obvious result.

\section{Toolkit for adaptive approximation}
\label{sec:tools}

The numerical problem under consideration is $\ell_p^m \embed \ell_q^m$ for $1 \leq p < q < \infty$ and~$m \in \N$.
The adaptive approximation scheme described in this paper
will identify the most important entries of the given vector $\vec{x} \in \R^m$ and measure them directly.
In precise terms, we determine a set $K \subseteq [m] := \{1,\ldots,m\}$ and yield an output $\vec{z} = \vec{x}^\ast_K \in \R^m$ with
\begin{equation} \label{eq:x_K*}
    z_j := \begin{cases}
        x_j &\text{for } j \in K, \\
        0 &\text{else.}
    \end{cases}
\end{equation}
How do we come up with a set~$K$?
We start by splitting $[m]$ into smaller sets~$J_d$, $d \in [D]$, so-called buckets, see Section~\ref{sec:hash}.
From each set~$J_d$ we then adaptively find one element to be included in $K$, see Section~\ref{sec:spot}.
The precise combination of these two steps is studied in Section~\ref{sec:discover}.
Further, Section~\ref{sec:precond} adds a step that allows for much larger buckets $J_d$ to begin with,
reducing this bucket to a significantly smaller set~$S_d \subseteq J_d$ from which we are to identify one supposedly important element.
The whole process will be repeated several times in the final approximation algorithm, see Section~\ref{sec:MSF}.

\textbf{Notation:} Given a vector $\vec{x} = (x_j)_{j\in [m]} \in \R^m = \R^{[m]}$ and a set $J \subset [m]$ we define the sub-vector $\vec{x}_J := (x_j)_{j \in J} \in \R^J$.

\subsection{Hashing}
\label{sec:hash}

Let $D \in \N$ and $\vec{H} = (H_i)_{i=1}^m$ be a family of random variables with values in $[D]$.
($\vec{H}$ is a so-called \emph{hash function} or \emph{hash family}.)
This defines disjoint (random) buckets
\begin{equation*}
    J_d = J_d^{\vec{H}} := \{i \in [m] \colon H_i = d\} \subseteq [m]  \,.
\end{equation*}
For $j \in [m]$ we denote by $B_j$ the bucket containing $j$, that is,
\begin{equation} \label{eq:B_j}
    B_j = B_j^{\vec{H}} := \{ i \in [m] : H_j = H_i \}  \,.
\end{equation}
Usually, $\vec{H}$ is chosen with pairwise independent values $H_i \sim \Uniform[D]$, each uniformly distributed on $[D]$. Pairwise independence ensures $\P(H_i=H_j) = \frac{1}{D}$ for $i\neq j$, 
allowing for a probabilistic bound on the norm of the sub-vector $\vec{x}_{B_j \setminus \{j\}}$,
see~\cite[Lemma~2.1]{KW24b}.
For the precise asymptotic error decay of the algorithm presented in this paper,
it will be also necessary to bound the cardinality
of the bucket~$B_j$, but with pairwise independent hashing we have some uncertainty therein.
Alternatively, we can define a random hashing where each hash value $d \in [D]$ occurs roughly the same number of times.

\begin{definition}
Let $D,m \in \mathbb{N}$
and let $\pi \sim \Uniform(\Sym(m))$ be a uniformly chosen random permutation of $[m]$. Consider the random vector
\begin{equation*}
  \vec{H} :=  \left(\left\lceil \frac{\pi(i) \cdot D}{m} \right\rceil\right)_{i=1}^m .
\end{equation*}
We call the distribution of $\vec{H}$ the equi-hash distribution with parameters $m,D$ and denote it by $\EquiHash(m,D)$.
\end{definition}

Note that a random vector distributed according to $\EquiHash(m,D)$ takes values in $[D]^m$ and satisfies two crucial properties:
\begin{enumerate}
    \item Each entry value appears either $\lfloor m/D\rfloor$ or $\lceil m/D \rceil$ times.
        In particular, a bucket $B$ obtained via $\vec{H} \sim \EquiHash(m,D)$ satisfies
        \begin{equation*}
            \#B \leq 1 + \frac{m}{D} \,.
        \end{equation*}
    \item The entries are pairwise negatively dependent in the sense that for $i \neq j$ one has
        \begin{equation*}
            \P(H_i = H_j) \leq \frac{1}{D} \,.
        \end{equation*}
\end{enumerate}
The statement of~\cite[Lemma~2.1]{KW24b} and its proof transfer to hashing by $\EquiHash$.
For the convenience of the reader we restate it now as Lemma \ref{lem:hash}.

\begin{lemma} \label{lem:hash}
    Let $\vec{x} \in \R^m$, $j \in [m]$, $1 \leq p < \infty$, and 
    $\alpha \in (0,1)$.
    If buckets $B_j$, see~\eqref{eq:B_j},
    are generated by a hash vector $\vec{H}$ with $\P(H_i = H_j) \leq \frac{1}{D}$
    for~\mbox{$i\neq j$}, then we have the probabilistic bound
    \begin{equation*}
        \P\left(\left\|\vec{x}_{B_j \setminus \{j\}}\right\|_{p}
                > \frac{\| \vec{x}_{[m] \setminus \{j\}  } \|_p}{(\alpha D)^{1/p}}\right)
            \leq \alpha \,.
            \label{eq:p-norm-hash}
    \end{equation*}
    This applies in particular to
    \begin{itemize}
        \item hash vectors $\vec{H}$ with pairwise independent entries $H_i \sim \Uniform[D]$,
        \item hash vectors $\vec{H} \sim \EquiHash(m,D)$.
            In this case $\#B_j \leq \left\lceil\frac{m}{D}\right\rceil$.
    \end{itemize}
\end{lemma}

The analysis of random measurements naturally leads to estimates that involve the $\ell_2$-norm.
In particular, by hashing we aim to isolate important coordinates $x_j$ in a way that,
with sufficiently high probability,
it satisfies a so-called \emph{heavy-hitter condition} with \emph{heavy-hitter constant}~$\gamma > 1$ on the corresponding bucket $B_j$:
\begin{equation} \label{eq:heavyhitter-abstract}
    \left\|\vec{x}_{B_j\setminus\{j\}}\right\|_2 \leq \frac{|x_j|}{\gamma} \,.
\end{equation}
For this requirement the appropriate choice of $D$ is as follows.

\begin{corollary} \label{cor:hash}
    Let $1 \leq p < \infty$
    and $\vec{x} \in \R^m$ with $\|\vec{x}\|_p \leq 1$.
    Further let $\gamma > 1$, $\eps,\delta_0 \in (0,1)$, and assume that $|x_j| \geq \eps$.
    If we take
    \begin{equation*}
        D := \begin{cases}
                \displaystyle
                \left\lceil \left(\gamma/\eps\right)^p \cdot \delta_0^{-1} \right\rceil
                    &\text{for } 1 \leq p \leq 2, \vspace{2pt}\\
                \displaystyle
                \left\lceil m^{1 - 2/p} \cdot \left(\gamma/\eps\right)^2 \cdot \delta_0^{-1} \right\rceil
                    &\text{for } 2 < p < \infty,
            \end{cases}
    \end{equation*}
    and draw a hash vector $\vec{H}$ as in Lemma~\ref{lem:hash},
    then
    \begin{equation*}
        \P\left(\left\|\vec{x}_{B_j\setminus\{j\}}\right\|_2 \leq \frac{|x_j|}{\gamma}\right)
            \geq 1 - \delta_0 \,.
    \end{equation*}
    Moreover, with 
    $\vec{H} \sim \EquiHash(m,D)$
    we have
    \begin{equation*}
        \#B_j \leq \begin{cases}
                        \left\lceil m \cdot (\eps/\gamma)^p \cdot \delta_0 \right\rceil
                            &\text{for } 1 \leq p \leq 2, \vspace{2pt} \\
                        \left\lceil m^{2/p} \cdot (\eps/\gamma)^2 \cdot \delta_0 \right\rceil
                            &\text{for } p > 2.
                    \end{cases}
    \end{equation*}
\end{corollary}
\begin{proof}
    We apply Lemma~\ref{lem:hash} with $\alpha = \delta_0$ to guarantee $\left\|\vec{x}_{B_j\setminus\{j\}}\right\|_2 \leq \frac{|x_j|}{\gamma}$ with probability $1-\delta_0$.
    First, in the case $1 \leq p \leq 2$ we have $\|\vec{x}_{B_j \setminus \{j\}} \|_2 \leq \|\vec{x}_{B_j \setminus \{j\}} \|_p$, and the choice of $D$ with $\|\vec{x}_{[m] \setminus\{j\}}\|_p \leq \|\vec{x}\|_p \leq 1$ leads to a probabilistic guarantee for
    \begin{equation*}
        \|\vec{x}_{B_j \setminus \{j\}} \|_p \leq \frac{1}{(\delta_0 D)^{1/p}} \leq \frac{\eps}{\gamma} \,.
    \end{equation*}
    In the case $p > 2$ we use $\|\vec{x}\|_2 \leq m^{\frac{1}{2} - \frac{1}{p}} \|\vec{x}\|_p$ for $\vec{x} \in \R^m$. 
    The choice of $D$ and $\|\vec{x}\|_p \leq 1$ leads to the probabilistic bound
    \begin{equation*}
        \|\vec{x}_{B_j \setminus \{j\}} \|_2 \leq \frac{m^{\frac{1}{2} - \frac{1}{p}}}{\sqrt{\delta_0 D}} \leq \frac{\eps}{\gamma} \,.
    \end{equation*}
    This shows the assertion.
\end{proof}

\begin{remark}
    While a hash vector $\vec{H} \sim \EquiHash(m,D)$ gives desirable theoretical guarantees,
    generating a hash vector $\vec{H}$ with pairwise independent entries might be cheaper
    (e.g.\ taking fewer random bits).
    With pairwise independent entries we have $\expect [\#B_j] = 1 + \frac{m-1}{D}$,
    but we would need a probabilistic guarantee for the cardinality $\#B_j$ to work with.
    Applying Lemma~\ref{lem:hash} to the vector $(1,\ldots,1)$ with $p=1$, we find an estimate for the cardinality of the bucket $B_j$:
    \begin{equation*}
        \P\left( \# B_j > 1 + \frac{m-1}{\alpha D} \right) \leq \alpha \,.
    \end{equation*}
    With fully independent hashing $\vec{H} \sim \Uniform([D]^m)$, even better estimates are possible as $\#B_j$ concentrates around its expectation.
    
    In any event, we can use Lemma~\ref{lem:hash} with failure probability $\alpha = \delta_0/2$
    for estimating the $p$-norm of the sub-vector and the cardinality of $B_j$.
    By a union bound, the probability that both estimates hold is at least $1-\delta_0$.
    In detail, in the case of $1 \leq p \leq 2$, if we take
    \begin{equation*}
        D := \left\lceil \left(\frac{\gamma}{\eps}\right)^p \cdot \frac{2}{\delta_0} \right\rceil ,
    \end{equation*}
    and draw a hash vector~$\vec{H}$ with pairwise independent entries $H_i \sim \Uniform[D]$, then
    for $\vec{x}$ with $x_j$ as in the assumptions of Corollary~\ref{cor:hash},
    \begin{equation*}
        \P\left(\left\|\vec{x}_{B_j\setminus\{j\}}\right\|_2 \leq \frac{|x_j|}{\gamma}
            \quad\text{and}\quad \#B_j \leq 1 + (m-1) \cdot \left(\eps/\gamma\right)^p\right)
            \geq 1 - \delta_0 \,.
    \end{equation*}
\end{remark}

\subsection{Spotting a single heavy hitter}
\label{sec:spot}

Details of the adaptive routine $\spot$ described in this section can be found in~\cite[Sec~2.3]{KW24b},
here we only describe essential aspects to address small adjustments necessary in our context.
The original idea of the algorithm stems from~\cite[Sec~3.1]{IPW11}.

Having split the domain $[m]$ into buckets~$J_d$, $d \in [D]$, on each bucket $J = J_d$ we run a routine $\spot_{\delta_2,k^\ast}(\vec{x},J)$ to identify the most important coordinate $j \in J$.
The routine succeeds in detecting a single coordinate $j \in J$
with probability at least $1-\delta_2$,
provided this coordinate satisfies the heavy hitter condition
\begin{equation} \label{eq:spotHHC}
    \|\vec{x}_{J \setminus \{j\}}\|_2
        \leq \frac{|x_j| \cdot \delta_2}{ 129\sqrt{2\log \frac{16}{\delta_2}}}\,,
\end{equation}
see~\cite[Lem~2.9 with Rem~2.10]{KW24b}.
This routine produces a nested sequence of sets
$J = S_0 \supseteq S_1 \supseteq \ldots \supseteq S_{k^\ast}$
via iterated shrinking
where each shrinking step $S_k \leadsto S_{k+1}$ is based on a subdivision of $S_k$ into up to
\begin{equation} \label{eq:D_k}
    D_k = D_k(\delta_2) := \left\lceil 2^{4 \cdot (9/8)^k + k+ 2} \delta_2^{-1} \right\rceil > 2^{4 \cdot (9/8)^k}
\end{equation}
smaller buckets via a random hash vector $\vec{H}^{(k)}$ that is independent of previous hash vectors (here, hashing with pairwise independent entries will do the job). We identify a small bucket $S_{k+1} := \shrink_{\vec{H}^{(k)}}(\vec{x},S_k)$ via a shrinking subroutine that takes two randomized measurements of $\vec{x}$ and with high probability ensures $j \in S_{k+1}$ for $k=0,\ldots,k^\ast - 1$.
If we obtain $\# S_k = 1$ or $S_k = \emptyset$ for some $k \leq k^\ast$, then this is the output of $\spot$,
otherwise, in a final step we yield the set
\begin{equation*}
    \spot_{\delta_2,k^\ast} (\vec{x},J) := \shrink_{\vec{h}^\ast}(\vec{x},S_{k^\ast})
\end{equation*}
where $\vec{h}^\ast$ provides a trivial hashing of $S_{k^\ast}$ into one-element sets.
In~\cite[eq~(11)]{KW24b} we chose a stopping index $k^\ast = k^\ast(m)$
that guaranteed $D_{k^\ast} \geq m$,
hence, the hash vector \mbox{$\vec{h}^\ast := (1,\ldots,m)$}
was suitable for the final shrinkage step.
Here now, in order to reduce the cost of the algorithm 
we exploit that after initial hashing of~$[m]$ via $\EquiHash(m,D)$,
we control the size of the bucket, $\# J \leq \left\lceil \frac{m}{D} \right\rceil$, which is significantly smaller than $m$. 
Hence, a hash vector $\vec{h}^\ast$ that enumerates the elements of $S_{k^\ast}$ will exist already
if we ensure $D_{k^\ast} \geq \left\lceil \frac{m}{D} \right\rceil$, namely, by~\eqref{eq:D_k} it suffices to put
\begin{align}
    k^\ast = k^\ast(m/D)
        &:= \max\left\{0,\, \left\lceil \log_{\frac98} \frac{\log_2 \left\lceil \frac{m}{D} \right\rceil}{8} \right\rceil 
                \right\} 
        \label{eq:k*}\\ 
    &\simeq \underbrace{\frac{1}{\log\frac{9}{4}}}_{\approx 8.4902} \cdot \log \log \frac{m}{D}
    \qquad\left(\text{for } \frac{m}{D} \to \infty\right). \nonumber
\end{align}
Since the choice of $k^\ast$ does not only depend on $m$,
in this paper we include the parameter $k^\ast$ in the description of $\spot$ in contrast to the notation in~\cite{KW24b}.
By construction, $\spot_{\delta_2,k^\ast}(\vec{x},J)$ will be a one-element set
(or the empty set in case of failure), and according to the analysis in \cite[Sec~2.3]{KW24b},
if~\eqref{eq:spotHHC} holds then we have
\begin{equation*}
    \P\Bigl(\spot_{\delta_2,k^\ast}(\vec{x},J) = \{j\}\Bigr) \geq 1 - \delta_2 \,.
\end{equation*}
The overall information cost of $\spot$ is bounded by
\begin{equation} \label{eq:n*}
    n^\ast = 2(k^\ast(m/D) + 1)
        \asymp \log \log \frac{m}{D}
        \qquad\text{for } m\geq 16D.
\end{equation}
Obviously, $n^\ast \geq 2$. On the other hand, recall that $m \geq 16D$ ensures $\log \log \frac{m}{D} > 1$.

\begin{remark}
    If we use hashing of $[m]$ into buckets $J_1,\ldots,J_D$ with pairwise independent hash values rather than using $\EquiHash(m,D)$,
    then we only have a stochastic guarantee that $\#J_d \leq 1 + \frac{2(m-1)}{\delta_0 D}$.
    Instead of $k^\ast = k^\ast(m/D)$ we choose
    \begin{equation*}
        k^\ast 
            = k^\ast\left(1 + \frac{2(m-1)}{\delta_0 D}\right)
            = \max\left\{0,\, \left\lceil \log_{\frac98} \frac{\log_2 \left\lceil 1 + \frac{2(m-1)}{\delta_0 D} \right\rceil}{4} \right\rceil \right\},
    \end{equation*}
    deterministically bounding the cost of $\spot$,
    while with sufficient probability we have $\#J_d \leq D_{k^\ast}$
    in which case we can rely on the probabilistic guarantees for $\spot$. 
    If we happen to end up with $D^\ast := \# S_{k^\ast} > D_{k^\ast}$,
    the attempt is considered a failure.
    In that case we may still perform a last shrinking step with an injective hashing
    $\vec{h}^\ast \in [D^\ast]^{S_{k^\ast}}$
    without any probabilistic guarantee of recovery whatsoever,
    or, alternatively, we may simply return the empty set.
\end{remark}

\subsection{Discovering important entries -- simple version}
\label{sec:discover}

We combine the previous two steps in the spirit of Hash-and-Recover by Woodruff et al.~\cite[Sec~E.2.2]{LNW17}
to form an algorithm that will (in expectation)
\emph{discover} around half 
of the important coordinates when being run once.
In Section~\ref{sec:MSF} we will see how several independent executions of $\discover$ lead to a set of coordinates that---when measured directly---provide an approximation with small expected error.

For a suitable choice of $D \in \N$ 
we pick a hash vector~$\vec{H} \sim \EquiHash(m,D)$,
resulting in a decomposition $J_1,\ldots,J_D$ of $[m]$, see Section~\ref{sec:hash}. 
For each bucket $J_d$ we apply the adaptive $\spot$ algorithm 
with parameters $\delta_2 = \frac{1}{3}$
and \mbox{$k^\ast = k^\ast(m/D)$}, see Section~\ref{sec:spot} and \eqref{eq:k*}.
Each instance of $\spot$ has the information cost \mbox{$n^\ast = 2(k^\ast+1)$} as described in~\eqref{eq:n*},
moreover, its random parameters are independent of the initial hash vector~$\vec{H}$.
In a simple version with only these two stages, our coordinate finding
algorithm for $\vec{x} \in \R^m$ returns the set
\begin{equation} \label{eq:Discover0}
    \discover_D^0(\vec{x}) := \bigcup_{d = 1}^{D} \spot_{\delta_2,k^\ast}(\vec{x},J_d)
\end{equation}
with at most $D$ elements.
The superindex~$0$ indicates that we are talking about the basic version of $\discover$ without preconditioning, see Section~\ref{sec:precond} for more details.
The choice of $D$ depends on the \emph{sensitivity} $\eps > 0$ we are interested in,
that means, coordinates $j$ with $|x_j| \geq \eps$ shall exhibit at least a $50\%$ chance of being detected by $\discover$. The precise value for $D$ is given in the following Lemma.

\begin{lemma} \label{lem:discover0}
    Let $m \in \N$, $1 \leq p < \infty$, $\vec{x} \in \R^m$ with $\|\vec{x}\|_p \leq 1$, and $\eps \in (0,1)$.
    If we take
    \begin{equation*}
        D := \begin{cases}
                \left\lceil 4 \cdot \left(387 \sqrt{2 \log 48}\right)^p \cdot \eps^{-p} \right\rceil
                    &\text{for } 1 \leq p \leq 2, \vspace{2pt}\\
                \left\lceil 1\,198\,152 \cdot \log 48 \cdot m^{1 - 2/p} \cdot \eps^{-2} \right\rceil
                    &\text{for } p > 2,
            \end{cases} 
    \end{equation*}
    and perform $\spot_{\delta_2,k^\ast}$ with $\delta_2 = \frac{1}{3}$
    and iteration depth~$k^\ast$ as in~\eqref{eq:k*}, namely,
    \begin{equation*}
        k^\ast(m/D) \asymp \log \log \frac{m}{D} \simeq \log \log (m \eps^p) \,,
        \qquad \text{for $m \geq 16 D$ or $m\eps^p \to \infty$,}
    \end{equation*}
    then for every 
    coordinate $j \in [m]$ with $|x_j| \geq \eps$ we have
    \begin{equation*}
        \P\bigl(j \notin \discover_D^0(\vec{x})\bigr) \leq \frac{1}{2} \,.
    \end{equation*}
    The information cost of $\discover_D^0$ is bounded by
    \begin{align*}
        \card(\discover^0_D)
            &\leq D \cdot 2\bigl(k^\ast(m/D) + 1\bigr) \\
            &\asymp \begin{cases}
                        \eps^{-p} \cdot \log \log (m \eps^p)
                            &\text{for } 1 \leq p \leq 2\,,\\
                        m^{1-2/p} \cdot \eps^{-2} \cdot \log \log (m \eps^p)
                            &\text{for } p > 2,
                    \end{cases}
    \end{align*}
    where the asymptotic equivalence holds for $m\eps^p \geq 16$.
\end{lemma}
\begin{proof}
    If an execution of $\spot_{\delta_2,k^\ast}$ 
    shall have failure probability at most $\delta_2 = \frac{1}{3}$ 
    for detecting a coordinate $j \in [m]$ with $|x_j| \geq \eps$,
    we require $j$ to fulfil a heavy hitter condition~\eqref{eq:heavyhitter-abstract}
    with heavy hitter constant
    \begin{equation*}
        \gamma = \frac{1}{\delta_2} \cdot 129 \sqrt{2 \log\frac{16}{\delta_2}}
            = 387 \cdot \sqrt{2 \log 48} \approx 1076.8 \,,
    \end{equation*}
    see \eqref{eq:spotHHC}.
    Hashing shall provide this condition with failure probability at most \mbox{$\delta_0 = \frac{1}{4}$},
    which by Corollary~\ref{cor:hash} leads to the choice of $D$ as stated in the lemma.
    In this setup the success for detecting a coordinate $j \in [m]$ with $|x_j| \geq \eps$
    can be computed as
    \begin{align*}
        \P\bigl(j &\in \discover_D^0(\vec{x})\bigr) \\
            &\geq \P\left( \|\vec{x}_{B_j \setminus \{j\}}\|_2 \leq \frac{\eps}{\gamma}\right)
                \cdot \P\left( \spot_{\delta_2,k^\ast}(\vec{x},B_j) = \{j\} \;\middle|\; \|\vec{x}_{B_j \setminus \{j\}}\|_2 \leq \frac{\eps}{\gamma}\right) \\
            &\geq (1-\delta_0) \cdot (1-\delta_2) = \frac{3}{4} \cdot \frac{2}{3} = \frac{1}{2} \,.
    \end{align*}
    Finally, observe $\frac{m}{D} \asymp m \eps^p$ for $1 \leq p \leq 2$, and $\frac{m}{D} \asymp m^{2/p} \eps^2 = \left(m \eps^p\right)^{2/p}$ for $p > 2$,
    compare the bound for the bucket size~$\#B_j$ in Corollary~\ref{cor:hash}.
    In any event, within
    the cost bound we find $\log \frac{m}{D} \asymp \log (m\eps^p)$.
\end{proof}

\subsection{Preconditioning}
\label{sec:precond}

The constant for $D$ in Lemma~\ref{lem:discover0} is quite big
(namely, $D \approx 4307\, \eps^{-1}$ for $p=1$, or $D \approx 4\,638\,287 \, \eps^{-2}$ for $p=2$).
Improving upon this constant is irrelevant for weak asymptotic error and complexity bounds, yet smaller constants are desirable because the routine $\discover$ requires $D$ instances of $\spot$, hence the constant is roughly proportional to the total cost.
The sole aim of this section is to provide a preconditioning routine that helps to reduce the constant for $D$.
If one is interested only in the asymptotic behaviour
of the error,
one may skip this section and directly proceed with Section~\ref{sec:Approx}.

We review~\cite[Lem~49]{LNW18} without using coding theory in our proof.
Starting from a relatively moderate heavy-hitter condition~\eqref{eq:heavyhitter-abstract} on a set $J \subseteq [m]$ with heavy-hitter constant $\gamma_0 = \sqrt{5}$,
we may reduce the candidate set by a so-called preconditioning procedure that uses $k$ independent measurements 
where $k$ is a parameter of this routine.
Namely, we use a Rademacher measurement matrix
\begin{equation*}
    A = (a_{ij})_{\substack{i=1,\ldots,k\\j \in J}} \in \R^{k \times J},
    \qquad a_{ij} \iid \Uniform\{\pm 1\} \,,
\end{equation*}
and we only retain the signs of these measurements:
\begin{equation*}
    \vec{s} = (s_i)_{i=1}^{k} := \sgn(A\vec{x}),
    \qquad\text{that is,}\quad
    s_i = \sgn\left(\sum_{j \in J} a_{ij} x_{j}\right) \in \{\pm 1\} \,.
\end{equation*}
(We put $\sgn(0) = 1$ at the price of asymmetry.)
For vectors $\vec{a},\vec{b} \in \{\pm1\}^k$ we consider the Hamming distance
\begin{equation*}
    d_H(\vec{a},\vec{b}) := \frac{1}{2} \|\vec{a}-\vec{b}\|_{1} \,.
\end{equation*}
Let $\vec{a}_j := (a_{ij})_{i=1}^k$ denote the $j$-th column of $A$.
We then define the output of the preconditioning algorithm as follows:
\begin{equation*}
    \precond_{k}(\vec{x},J)
        := \left\{j \in J \mid d_H(\vec{a}_j,\vec{s})\leq \frac{k}{6} \;\vee\;
                                d_H(\vec{a}_j,-\vec{s})\leq \frac{k}{6}
            \right\} .
\end{equation*}

\begin{lemma} \label{lem:precond}
    For $\vec{x} \in \R^m$ and $j^\ast \in J \subset [m]$ assume the heavy hitter condition
    \begin{equation*} \label{eq:weakHHC}
        \|\vec{x}_{J \setminus\{j^\ast\}}\|_2
            \leq \frac{|x_{j^\ast}|}{\sqrt{5}} \,.
    \end{equation*}
    Let $\gamma > 1$ and $\delta_1 \in (0,1)$.
    If we choose
    \begin{equation*}
        k := \left\lceil 36 \log\left(\frac{1+\frac{2}{5}\gamma^2}{\delta_1}\right)\right\rceil ,
    \end{equation*}
    then, with $S := \precond_{k}(\vec{x},J) \subseteq J$, we have
    \begin{equation*}
        \P\left( j^\ast \in S \quad \text{and} \quad
            \|\vec{x}_{S \setminus\{j^\ast\}}\|_2 \leq \frac{|x_{j^\ast}|}{\gamma}
            \right)
            \geq 1-\delta_1 \,.
    \end{equation*}
\end{lemma}
\begin{proof}
    The idea is that $\vec{s} \approx \sgn(x_{j^\ast}) \cdot \vec{a}_{j^\ast}$
    in the sense that likely only few entries of these two vectors differ.
    In fact, if
    \begin{equation*}
        Y_i := \sum_{j \in J \setminus\{j^\ast\}} 
                    \frac{x_j}{x_{j^\ast}a_{i,j^\ast}} \cdot a_{ij}
            > -1
    \end{equation*}
    holds then $s_i = \sgn(x_{j^\ast}) \cdot a_{i,j^\ast}$.
    Here, $Y_i$ is the sum of independent random variables with square-summable absolute bounds $\frac{|x_j|}{|x_{j^\ast}|}$. By Hoeffding's inequality we have
    \begin{align*}
        \P\bigl(s_i \neq \sgn(x_{j^\ast} a_{i,j^\ast})\bigr)
            &\leq \P(Y_i \leq -1)
            \leq \exp\left(-\frac{x_{j^\ast}^2}{2\|\vec{x}_{J\setminus\{j^\ast\}}\|_2^2}\right)
            \\
            &\leq \exp\left(-\frac{5}{2}\right) < \frac{1}{12} \,.
    \end{align*}
    It follows that
    \begin{equation*}
        \mu := \expect \left[d_H(\vec{s}, \sgn(x_{j^\ast}) \cdot \vec{a}_{j^\ast})\right]
            < \frac{k}{12} \,.
    \end{equation*}
    Since 
    $\frac{1}{2}|s_i - \sgn(x_{j^\ast}) \cdot a_{i,j^\ast}| \in \{0,1\}$
    are i.i.d.\ Bernoulli random variables,
    we may use a Chernoff bound for binomially distributed random variables,
    namely, with $\delta = \frac{k}{6\mu} - 1 \geq 1$ we find
    \begin{align*}
        \P\left(d_H(\vec{s}, \sgn(x_{j^\ast}) \cdot \vec{a}_{j^\ast})
                > \frac{k}{6}\right)
            &= \P\bigl(d_H(\vec{s}, \sgn(x_{j^\ast}) \cdot \vec{a}_{j^\ast})
                > (1+\delta)\mu\bigr) \\
            &\leq \exp\left(- \frac{\min\{\delta,\delta^2\} \cdot \mu}{3}\right)
            = \exp\left(- \frac{k/6 - \mu}{3}\right) \\
            &\leq \exp\left(-\frac{k}{36}\right) =: \alpha \,.
    \end{align*}
    This gives a guarantee for correctly identifying a subset of~$J$ that still contains the important coordinate $j^\ast$:
    \begin{align}
        \P\bigl(j^\ast \in \precond_k(\vec{x},J)\bigr)
            &\geq \P\left(d_H(\vec{s}, \sgn(x_{j^\ast}) \cdot \vec{a}_{j^\ast})
                \leq \frac{k}{6}\right)
                \nonumber \\
            &\geq 1 - \exp\left(-\frac{k}{36}\right)
            = 1 - \alpha \,.
            \label{eq:precondtrue}
    \end{align}

    Now suppose that $j^\ast \in S = \precond_k(\vec{x},J)$. Then, by the triangle inequality we further know
    \begin{equation*}
        S \subseteq \left\{ j \in J \colon \min\{d_H(\vec{a}_j,\vec{a}_{j^\ast}),d_H(\vec{a}_j,-\vec{a}_{j^\ast})\}
                \leq \frac{k}{3} \right\} =: \overline{S} \,.
    \end{equation*}
    For $j \in J \setminus\{j^\ast\}$, the quantity $d_H(\vec{a}_j,\vec{a}_{j^\ast})$ follows a symmetrical binomial distribution with expectation $\mu := \expect\left[d_H(\vec{a}_j,\vec{a}_{j^\ast})\right] = \frac{k}{2}$.
    Another Chernoff bound,
    with $\delta = \frac{1}{3}$,
    gives
    \begin{align*}
        \P\left(d_H(\vec{a}_j,\vec{a}_{j^\ast}) \leq \frac{k}{3}\right)
            &= \P\bigl(d_H(\vec{a}_j,\vec{a}_{j^\ast}) \leq (1-\delta) \mu\bigr) \\
            &\leq \exp\left(-\frac{\delta^2 \mu}{2}\right)
            = \exp\left(-\frac{k}{36}\right) .
    \end{align*}
    We can use this to estimate the mean squared norm of these coordinates:
    \begin{align*}
        \expect \|\vec{x}_{\overline{S} \setminus\{j^\ast\}}\|_2^2
            &\leq \sum_{j \in J \setminus\{j^\ast\}} \P\left(d_H(\vec{a}_j,\vec{a}_{j^\ast}) \leq \frac{k}{3} \text{ or } d_H(\vec{a}_j,-\vec{a}_{j^\ast}) \leq \frac{k}{3}\right) \cdot |x_j|^2 \\
            &\leq 2\exp\left(-\frac{k}{36}\right)
                \cdot \|\vec{x}_{J \setminus \{j^\ast\}}\|_2^2 \,.
    \end{align*}
    By Markov's inequality we find for $\gamma > 1$:
    \begin{align}
        \P\left( \|\vec{x}_{\overline{S} \setminus\{j^\ast\}}\|_2 > \frac{|x_{j^\ast}|}{\gamma}\right)
            &\leq \P\left( \|\vec{x}_{\overline{S} \setminus\{j^\ast\}}\|_2 
                    > \frac{\sqrt{5} \cdot \|\vec{x}_{J \setminus\{j^\ast\}}\|_2}{\gamma}\right)
            \nonumber \\
            &\leq \frac{2}{5} \, \gamma^2 \cdot \exp\left(-\frac{k}{36}\right) =: \beta\,.
            \label{eq:precondnorm}
    \end{align}

    Finally, combining the findings~\eqref{eq:precondtrue} and~\eqref{eq:precondnorm}, we obtain
    \begin{align*}
        \P&\left( j^\ast \in S \quad \text{and} \quad
            \|\vec{x}_{S \setminus\{j^\ast\}}\|_2 \leq \frac{|x_{j^\ast}|}{\gamma}
            \right) \\
            &\geq \P\left( j^\ast \in S \quad \text{and} \quad
            \|\vec{x}_{\overline{S} \setminus\{j^\ast\}}\|_2 \leq \frac{|x_{j^\ast}|}{\gamma}
            \right) \\
            &\geq 1 - \P\left( j^\ast \notin S\right)
                    - \P\left(\|\vec{x}_{\overline{S} \setminus\{j^\ast\}}\|_2 > \frac{|x_{j^\ast}|}{\gamma}\right) \\
            &\geq 1 - \alpha - \beta \,.
    \end{align*}
    The choice of $k$ in the lemma ensures
    $\alpha + \beta = \left(1 + \frac{2}{5} \, \gamma^2\right) \exp\left(-\frac{k}{36}\right) \leq \delta_1$.
\end{proof}

With the help of preconditioning we may define the following modification of the basic version of $\discover$, compare~\eqref{eq:Discover0}:
\begin{equation} \label{eq:discover+}
    \discover_D^+(\vec{x}) := \bigcup_{d = 1}^{D} \spot_{\delta_2,k^\ast}\bigl(\vec{x},\precond_k(J_d)\bigr) \,.
\end{equation}
Again, the three random components, namely hashing, $\spot$, and $\precond$ shall be independent.
The parameters $D$, $\delta_2$, $k^\ast$, and $k$ are to be chosen differently now.

\begin{lemma} \label{lem:discover+}
    Let $m \in \N$, $1 \leq p < \infty$, $\vec{x} \in \R^m$ with $\|\vec{x}\|_p \leq 1$, and $\eps \in (0,1)$.
    If we take
    \begin{equation*}
        D := \begin{cases}
                \left\lceil 6 \cdot 5^{p/2} \cdot \eps^{-p} \right\rceil
                    &\text{for } 1 \leq p \leq 2,\\
                \left\lceil 30 \cdot m^{1 - p/2} \cdot \eps^{-2} \right\rceil
                    &\text{for } p > 2,
            \end{cases}
    \end{equation*}
    and perform $\precond_k$ with $k=551$
    as well as $\spot_{\delta_2,k^\ast}$ with $\delta_2 = \frac{1}{4}$
    and iteration depth~$k^\ast$ as in~\eqref{eq:k*}, namely
    \begin{equation*}
        k^\ast(m/D) \asymp \log \log \frac{m}{D} \simeq \log \log (m \eps^p) \,,
    \end{equation*}
    then for every 
    coordinate $j \in [m]$ with $|x_j| \geq \eps$ we have
    \begin{equation*}
        \P\bigl(j \notin \discover_D^+(\vec{x})\bigr) \leq \frac{1}{2} \,.
    \end{equation*}
    The information cost of $\discover_D^+$ is bounded by
    \begin{align*}
        \card(\discover^+_D)
            &\leq D \cdot \bigl(553 + 2k^\ast(m/D)\bigr) \\
            &\asymp \begin{cases}
                        \eps^{-p} \cdot \log \log (m \eps^p)
                            &\text{for } 1 \leq p \leq 2,\\
                        m^{1-2/p} \cdot \eps^{-2} \cdot \log \log (m \eps^p)
                            &\text{for } p > 2.
                    \end{cases}
    \end{align*}
\end{lemma}
\begin{proof}
    We choose failure probabilities $\delta_0 = \frac16$, $\delta_1 = \frac15$, and $\delta_2 = \frac14$ such that we can bound the success probability for any given heavy hitter by multiplying the success probabilities of hashing, preconditioning, and spotting:
    $(1-\delta_0)(1-\delta_1)(1-\delta_2) = \frac12$.
    In this setup, $\spot_{\delta_2,k^\ast}$ requires the heavy hitter condition~\eqref{eq:heavyhitter-abstract} with
    \begin{equation*}
        \gamma = \frac{1}{\delta_2} \cdot 129 \sqrt{2 \log \frac{16}{\delta_2}}
            = 516 \cdot \sqrt{2\log 64}
            \approx 1488.2 \,,
    \end{equation*}
    see~\eqref{eq:spotHHC}.
    From Lemma~\ref{lem:precond} we find the parameter choice
    \begin{equation*}
        k := \left\lceil 36 \log \left(\frac{1+\frac{2}{5} \gamma^2}{\delta_1}\right)
                \right\rceil
    \end{equation*}
    for $\precond_k$.
    Initial hashing, though, is cheaper now 
    since we choose $D$ according to Corollary~\ref{cor:hash} but with the much smaller heavy hitter constant $\gamma_0 = \sqrt{5}$.
    The choice of the stopping index $k^\ast$ for $\spot$ follows from~\eqref{eq:k*},
    the change due to the smaller constant is marginal because of the double logarithm.
\end{proof}

\begin{remark}
    With $k = 551$ we can estimate the expected cardinality of the preconditioned bucket 
    $S := \precond_k(\vec{x},B_j)$ where $j \in [m]$ is a fixed (important) coordinate:
    \begin{equation*}
        \expect [\#(S \setminus\{j\})]
            \leq 2 \exp\left(-\frac{551}{36}\right) \cdot \#(B_j\setminus\{j\})
            < \frac{\#(B_j\setminus\{j\})}{2\,218\,648} \,.
    \end{equation*}
    This means that if we have initial buckets with cardinality $\# B_j \leq 10^7$, then we will very likely end up with a preconditioned bucket~$S$ that contains only one element, so performing $\spot$ would not be needed anymore.
    If, however, $m$ is larger, spot is still relevant and we need to know the size of a preconditioned bucket $S$ to decide on the iteration depth $k^\ast$ of $\spot$.
    In the lemma we simply chose $k^\ast$ on the basis of $\# S \leq \#B_j \leq \lceil \frac{m}{D} \rceil$, but one could also take into account a much smaller probabilistic bound on $\# S$.
    In any event, the iteration depth of $\spot$ will be 
    $k^\ast \simeq \log_{\frac{9}{8}} \log (m\eps^{p})$.
    
    The advantage of preconditioning is most striking in the case $p = 2$:
    Here we roughly spend
    $551 \cdot 30 \, \eps^{-2} = 16\,530 \, \eps^{-2}$
    measurements in total
    for preconditioning on $D \approx 30\, \eps^{-2}$ buckets.
    This is still significantly smaller than the number of $D \approx 5 \cdot 10^6 \, \eps^{-2}$ buckets we would need to deal with in the basic version $\discover^0$
    without preconditioning.
\end{remark}

\begin{remark}
    Preconditioning could also be included in uniform approximation, i.e.\ the problem $\ell_p^m \embed \ell_\infty^m$, see the prior work~\cite{KW24b}.
    In contrast to $\ell_q$-approximation with $q<\infty$,
    for uniform approximation with expected error~$\eps$ we require that with probability $1 - \frac{\eps}{2}$
    all coordinates with $|x_j|>\frac{\eps}{2}$ are detected \emph{in one run} of the algorithm, provided $\|\vec{x}\|_p \leq 1$ for $1 \leq p \leq 2$.
    The modified method would have the following structure:
    \begin{enumerate}
        \item Hashing $[m]$ into $D$ buckets.
        \item Selecting $k \asymp \eps^{-p}$ important buckets.
        \item Optionally: Preconditioning of the selected buckets.
        \item Applying $\spot$ on each of the selected (preconditioned) buckets.
    \end{enumerate}
    Preconditioning will help to significantly reduce $D$,
    for instance, in case of $p=2$ we could take
    $D \simeq 4096 \, \eps^{-4}$
    instead of~$D \approx 7 \cdot 10^{10} \cdot \eps^{-14}$
    without preconditioning.
    However, the information cost of the bucket selection step is only roughly proportional to~$\log D$, and this is the only stage where $D$ affects the cost.
\end{remark}

\section{Adaptive randomized approximation}
\label{sec:Approx}

In Section~\ref{sec:MSF} we describe and analyse the final adaptive randomized algorithm
for finite-dimensional sequence space embeddings $\ell_p^m \embed \ell_q^m$.
In Section~\ref{sec:complexity} we draw conclusions for the complexity and error rates.

\subsection{A multi-sensitivity algorithm}
\label{sec:MSF}

Adaptive approximation of $\ell_p^m \embed\ell_q^m$ for
$1 \leq p < q < \infty$
is based on a repeated independent execution of $\discover$ with different hashing parameters $D$.
The particular approach we are taking directly corresponds to the first phase of~\cite[Alg~3]{LNW17}.
Our error criterion allows for this simplified approach and a straightforward analysis.
For $l \in \N$, indicating different sensitivity levels of the algorithm, define hashing parameters
\begin{equation} \label{eq:D^(l)}
    D^{(l)} := 
        \begin{cases}
            \left\lceil C_p \cdot 2^{l} \right\rceil
                &\text{for } 1 \leq p \leq 2, \\
            \left\lceil C_2 \cdot m^{1 - 2/p} \cdot 2^l \right\rceil
                &\text{for } p > 2,
        \end{cases}
\end{equation}
with the constant $C_p := 4 \cdot \left(387\sqrt{2\log48}\right)^p$
if we work with the basic version $\discover^0$, compare Lemma~\ref{lem:discover0},
or $C_p := 6 \cdot 5^{p/2}$ if we work with the version $\discover^+$ which features preconditioning, compare Lemma~\ref{lem:discover+}.
From now on, we simply write $\discover$ for any of the two versions.
The choice~\eqref{eq:D^(l)} of the hashing parameter~$D^{(l)}$
means that for vectors $\|\vec{x}\|_p \leq 1$, the routine $\discover_{D^{(l)}}$
achieves sensitivity
\begin{equation*}
    \eps_l := \begin{cases}
                    2^{-l/p} &\text{for } 1 \leq p \leq 2,\\
                    2^{-l/2} &\text{for } p > 2,
                \end{cases}
\end{equation*}
meaning that, in expectation, at least half of the coordinates with $|x_j| \geq \eps_l$ will be discovered.
The algorithm is governed by two parameters $L, R \in \N$ and first computes a set
\begin{equation*}
    K_{L,R} := \bigcup_{l=1}^L \bigcup_{r=1}^R \discover_{D^{(l)}}^{(r,l)}(\vec{x})\,,
\end{equation*}
where by the superscripts $(r,l) \in [R] \times [L]$ we indicate independent calls of $\discover$.
As a final output we return
\begin{equation*}
    A_{L,R}(\vec{x}) := \vec{x}_{K_{L,R}}^\ast
\end{equation*}
by $\# K_{L,R}$ direct evaluations of entries of~$\vec{x}$,
see~\eqref{eq:x_K*} for the definition of $\vec{x}^{\ast}_K$ given $K \subset [m]$.
The parameter $L$ is the number of sensitivity levels,
the parameter $R$ is the number of repetitions at each sensitivity level.
We may thus call the method a ``multi-sensitivity algorithm''.

\begin{theorem} \label{thm:A_LR}
    Let $m,L \in \N$, $1 \leq p < q < \infty$, and $\vec{x} \in \R^m$ with $\|\vec{x}\|_p \leq 1$.
    If we choose $R := \lceil q/\min\{2,p\} \rceil$,
    the algorithm $A_{L,R}$ achieves accuracy
    \begin{equation*}
        \left(\expect \|\vec{x} - A_{L,R}(\vec{x})\|_q^q\right)^{1/q}
            \leq 3^{1/q} \cdot \begin{cases}
                    2^{-\left(\frac{1}{p} - \frac{1}{q}\right) \cdot L}
                        &\text{for } 1 \leq p \leq 2,\\
                    2^{-\frac{1}{2}\left(1 - \frac{p}{q}\right) \cdot L}
                        &\text{for } p > 2.
                \end{cases}
    \end{equation*}
    The information cost is upper bounded by
    \begin{align*}
        \card A_{L,R}
            &\preceq
                \begin{cases}
                    \displaystyle 2^L \cdot \log \log \frac{m}{2^L}
                        &\text{for } 1 \leq p \leq 2, \\
                    \displaystyle m^{1-2/p} \cdot 2^L \cdot \log \log \frac{m^{2/p}}{2^L}
                        &\text{for } p > 2,
                \end{cases}
    \end{align*}
    the asymptotic relation holding for $m^{\max\{1,\,2/p\}} \geq 16 \cdot 2^L$.
\end{theorem}
\begin{proof}
    For convenience we will write $p' := \min\{2,p\}$ and $t_+ := \max\{0,t\}$ for $t \in \R$
    to accommodate for both regimes at once.
    We classify the coordinates according to the sensitivity levels $\eps_l = 2^{-l/p'}$
    by forming index sets
    \begin{equation*}
        I_l := \left\{ j \in [m] \colon \eps_l < |x_j| \leq \eps_{l-1}\right\} .
    \end{equation*}
    By definition of the sensitivity levels and $\|\vec{x}\|_p \leq 1$,
    we have \mbox{$\#I_l \leq \eps_l^{-p} = 2^{l\frac{p}{p'}}$}.
    Further, the hash parameter $D^{(l)}$ is chosen such that
    for all \mbox{$j \in I_1\cup\ldots \cup I_l$} we have
    \begin{equation*}
        \P\bigl( j \notin \discover_{D^{(l)}}(\vec{x}) \bigr) \leq \frac{1}{2} \,,
    \end{equation*}
    see Section~\ref{sec:discover}.
    In other words, for $1 \leq l_0 \leq L$ and $j \in I_{l_0}$ we know
    \begin{equation*}
        \P\bigl( j \notin K_{L,R} \bigr)
            \leq \prod_{l=l_0}^L \prod_{r=1}^R \P\left( j \notin \discover_{D^{(l)}}^{(r,l)}\right)
            \leq 2^{-(L - l_0 + 1) \cdot R} ,
    \end{equation*}
    hence, we expect the following cardinality for the set of entries that have not been discovered:
    \begin{equation*}
        \expect \left[\# (I_l \setminus K_{L,R})\right]
            \leq 2^{ l\frac{p}{p'} - (L - l + 1) \cdot R} .
    \end{equation*}
    We do not expect to recover any of the less important coordinates from the complementary set
    \begin{equation*}
        C_L := \{j \in [m] \colon |x_j| \leq \eps_L \}\,,
    \end{equation*}
    though it might happen for some. 
    The $q$-norm of $\vec{x}_{C_L}$ can be bounded by interpolation between the $\ell_p$- and the $\ell_\infty$-norm,
    namely,
    \begin{align}
        \frac{1}{q} &= \frac{\lambda}{p} + \frac{1-\lambda}{\infty}
        \quad\text{with } \lambda = \frac{p}{q} \in (0,1)\,, 
        \nonumber
        \\
        \|\vec{x}_{C_L}\|_q
            &\leq \|\vec{x}_{C_L}\|_p^\lambda \cdot \|\vec{x}_{C_L}\|_\infty^{1-\lambda}
            \leq 1^\lambda \cdot \eps_L^{1-\lambda}
            \leq \eps_L^{1-\frac{p}{q}}
            \leq 2^{-\frac{L}{p'}\left(1 - \frac{p}{q}\right)} , \label{eq:Interpolation}
    \end{align}
    see for instance \cite[Lem~2.4]{KNR19}.
    This leads to the desired error estimate
    \begin{align*}
        \expect \|\vec{x} - A_{L,R}(\vec{x})\|_q^q
            &\leq \|\vec{x}_{C_L}\|_q^q
                + \sum_{l=1}^L \expect\left[\#(I_l \setminus K_{L,R})\right] \cdot \eps_{l-1}^q \\
            &\leq 2^{-\frac{L}{p'}\left(q - p\right)}
                    + \sum_{l=1}^L 2^{ l \frac{p}{p'} - (L - l + 1) \cdot R} \cdot 2^{-(l-1) \cdot \frac{q}{p'}} \\
            &= 2^{-\frac{L}{p'}\left(q - p\right)}
                    \left(1 
                    + 2^{\frac{q}{p'} - R} \sum_{l=1}^L 2^{-\left(\frac{p}{p'} + R - \frac{q}{p'}\right) \cdot (L-l)}
                    \right).
    \end{align*}
    With $\frac{p}{p'} \geq 1$ and by the choice of $R = \lceil q/p' \rceil$ we have
    \begin{equation*}
        2^{\frac{q}{p'} - R} \sum_{l=1}^L 2^{-\left(\frac{p}{p'}+R - \frac{q}{p'}\right) \cdot (L-l)}
            \leq 1 \cdot \sum_{k=0}^\infty 2^{-k}
            = 2 \,,
    \end{equation*}
    hence,
    \begin{align*}
        \expect \|\vec{x} - A_{L,R}(\vec{x})\|_q^q
            & \leq 3 \cdot 2^{-\frac{L}{p'}\left(q - p\right)} .
    \end{align*}
    Taking this to the power of $\frac{1}{q}$ we find precisely the accuracy we claimed.

    We use Lemma~\ref{lem:discover0} or~\ref{lem:discover+}, respectively,
    to bound the combined cost for all calls of $\discover$ as long as $m \geq 16 D^{(L)}$,
    \begin{align}
        \sum_{l=1}^L R \cdot &\card\left(\discover_{D^{(l)}}\right)
            \preceq \sum_{l=1}^L D^{(l)} \cdot \log \log \frac{m}{D^{(l)}} 
                \nonumber \\
            &\asymp m^{(1-2/p)_+} \cdot \sum_{l=1}^L 2^l \cdot \log \log \frac{m^{p'\!/p}}{2^l} 
                \nonumber 
                \qquad\qquad \text{(using $(1-2/p)_+ = 1 - p'\!/p$)}\\
            &= m^{(1-2/p)_+} \cdot 2^L \log \log \frac{m^{p'\!/p}}{2^L} \cdot \sum_{l=1}^L 2^{-(L-l)} \cdot \frac{\log \left(\left(\log \frac{m^{p'\!/p}}{2^L}\right) + (L-l) \log 2\right)}{\log \log \frac{m^{p'\!/p}}{2^L}}
                \nonumber \\
            &\preceq m^{(1-2/p)_+} \cdot 2^L \log \log \frac{m^{p'\!/p}}{2^L}
                \cdot \sum_{k=0}^\infty 2^{-k} \cdot \log(2+k)
                \nonumber \\
            &\asymp m^{(1-2/p)_+} \cdot 2^L \log \log \frac{m^{p'\!/p}}{2^L} \,.
                \label{eq:loglog(m/2^L)}
    \end{align}
    The algorithm $A_{L,R}$ will finally return a vector with at most
    \begin{equation*}
        \# K_{L,R}
            \leq \sum_{l=1}^L R \cdot D^{(l)}
            = R \sum_{l=1}^L \lceil C_{p'} \cdot m^{(1-2/p)_+} \cdot 2^l \rceil
            \simeq C_p R \cdot m^{(1-2/p)_+} \cdot 2^{L+1}
    \end{equation*}
    directly measured coordinates and all other coordinates set to zero.
    The information cost of this final step
    is clearly dominated by~\eqref{eq:loglog(m/2^L)}
    which leads to the overall asymptotic cost bound as stated.
    Note that~\eqref{eq:loglog(m/2^L)} holds as a weak asymptotic upper bound of the cost
    even if we extend the domain of comparison to $m^{p'\!/p} \geq 16 \cdot 2^L$, see Example~\ref{ex:e(L,m)} for the case $1\leq p \leq 2$.
\end{proof}

\begin{remark}[Avoiding overlap]
    If we perform the instances of $\discover$ sequentially,
    we will find a growing sequence
    $\emptyset = K^{(0)} \subseteq \ldots \subseteq K^{(RL)} = K_{R,L}$
    of candidate sets,
    and it is natural to apply the $i$-th instance of $\discover$
    to the restricted vector $\vec{x}_{[m]\setminus\{K^{(i)}\}}$
    as it has been suggested in~\cite{LNW17}.
    With this modification we might find better approximations,
    for the error analysis, however, this seems not to be a necessary step.
\end{remark}

\begin{remark}[Homogeneity] \label{rem:homogen}
    Note that $A_{L,R}$ is homogeneous in the sense of
    $A_{L,R}(t \vec{x}) = t \cdot A_{L,R}(\vec{x})$ for any scalar $t \neq 0$,
    see~\cite{KK24} for a general reference on homogeneous algorithms.
    This is due to the fact that all adaption decisions in the scheme are always based on the ratio of measurements
    but not on absolute values.
    Homogeneity implies in particular that, for $1 \leq p \leq 2$ and all $\vec{x} \in \R^m$ we can state
    \begin{equation*}
        \expect \|\vec{x} - A_{L,R}(\vec{x})\|_q
            \leq 3^{1/q} \cdot 2^{-\left(\frac{1}{p} - \frac{1}{q}\right) \cdot L} \cdot \|\vec{x}\|_p \,.
    \end{equation*}
\end{remark}

\subsection{Complexity}
\label{sec:complexity}

\begin{theorem} \label{thm:complexity}
    Let $m \in \N$, $m \geq 16$, 
    $1 \leq p < q < \infty$.
    Then for 
    $\left(\frac{16}{m}\right)^{\frac{1}{p} - \frac{1}{q}} \leq \eps < 1$
    we have
    \begin{multline*}
        n^\ran(\eps,\ell_p^m \embed \ell_q^m) \\
            \preceq \begin{cases}
                \eps^{\left. - 1 \middle/ \left(\frac{1}{p} - \frac{1}{q}\right) \right.}
                \cdot \log \log \left(m \cdot \eps^{\left. 1 \middle/ \left(\frac{1}{p} - \frac{1}{q}\right) \right.}\right)
                    &\text{for } 1 \leq p \leq 2, \\
                m^{1 - \frac{2}{p}} \cdot \eps^{\left. -2 \middle/ \left(1 - \frac{p}{q}\right) \right.}
                \cdot \log \log \left(m \cdot \eps^{\left. 1 \middle/ \left(\frac{1}{p} - \frac{1}{q}\right) \right.}\right)
                    &\text{for } p > 2.
            \end{cases}
    \end{multline*}
    Conversely, for $n \in \N$ with $m \geq 16n$ we have
    \begin{equation*}
        e^\ran(n,\ell_p^m \embed \ell_q^m)
            \preceq
                \begin{cases}
                    \displaystyle
                    \min\left\{1, 
                    \left(\frac{\log \log \frac{m}{n}}{n}\right)^{\frac{1}{p} - \frac{1}{q}}
                        \right\}
                        &\text{for } 1 \leq p \leq 2,\\
                    \displaystyle 
                    \min\left\{1, 
                        \left(\frac{m^{1-2/p} \cdot \log \log \frac{m}{n}}{n}\right)^{\frac{1}{2}\left(1 - \frac{p}{q}\right)}
                        \right\}
                        &\text{for } p > 2.
                \end{cases}
    \end{equation*}
\end{theorem}
\begin{proof}
    From the $q$-moment error stated in Theorem~\ref{thm:A_LR} and Jensen's inequality we can directly conclude
    on the Monte Carlo error as defined in~\eqref{eq:e(A_n)}, namely
    \begin{equation*}
        e(A_{L,R}, \ell_p^m \embed \ell_q^m)
            \leq 3^{1/q} \cdot 2^{- \frac{L}{p'} \left(1 - \frac{p}{q}\right)},
    \end{equation*}
    for $L \in \N$ and $R = \lceil q/p' \rceil$ where we write $p' = \min\{2,p\}$ again.
    If this shall be smaller or equal than a given $\eps \in (0,1)$ then we need to choose
    \begin{equation*}
        L := \left\lceil \frac{\log_2 \frac{3^{1/q}}{\eps}}{\frac{1}{p'} \left(1 - \frac{p}{q}\right)} \right\rceil ,
    \end{equation*}
    which implies
    \begin{equation} \label{eq:eps-vs-2^L}
        \eps^{\left. -p' \middle/ \left(1 - \frac{p}{q}\right) \right.}
            \leq 2^L
            \leq 2 \cdot 3^{\left. p' \middle/ \left(q - p\right) \right.}
                    \eps^{\left. - p' \middle/ \left(1 - \frac{p}{q}\right) \right.} .
    \end{equation}
    This relation is independent of $m$ and holds for all $L \in \N$.
    (Overhashing with \mbox{$2^L \geq m$} would lead to error~$0$ because $\EquiHash(m,D)$ would then put all coordinates into single element buckets,
    hence all entries are measured and so the upper bound trivially holds.)
    From Theorem~\ref{thm:A_LR} we know that for $m \geq 16 \cdot 2^L$ and with a suitable constant $C > 1$ we have
    the following estimate, and by \eqref{eq:eps-vs-2^L} it can be bounded in terms of $\eps$
    (again with the notation $t_+ = \max\{0,t\}$):
    \begin{multline*}
        \card A_{L,R}
            \leq C \cdot m^{(1-2/p)_+} \cdot 2^L \cdot \log \log \frac{m^{p'\!/p}}{2^L} \\
            \leq \underbrace{C \cdot 2 \cdot 3^{\left. p' \middle/ \left(q-p\right) \right.}}_{C'}
                    \,\cdot\,m^{(1-2/p)_+} \cdot \eps^{\left. -p' \middle/ \left(1 - \frac{p}{q}\right) \right.}
                    \cdot \log \log \left(m\cdot \eps^{\left. 1 \middle/ \left(\frac{1}{p} - \frac{1}{q}\right) \right.}
                                    \right) .
    \end{multline*}
    (For $p > 2$ we also omitted the exponent $\frac{p'}{p} = \frac{2}{p} < 1$ in the argument of the double logarithm.)
    The constraint $m \cdot \eps^{\left. 1 \middle/ \left(\frac{1}{p} - \frac{1}{q}\right) \right.} \geq 16$
    is less restrictive than $m^{p'\!/p} \geq 16 \cdot 2^L$,
    but analogously to Example~\ref{ex:e(L,m)} we may extend the validity of the asymptotic estimate.

    We are now passing to the asymptotic $n$-th minimal error.
    By construction, the error of the method $A_{L,R}$ is always bounded by~$\|\vec{x}\|_p \leq 1$.
    Aiming for better bounds,
    we want to make sure that the cost does not exceed a given limit of $n \in \N$, hence:
    \begin{equation} \label{eq:L-vs-n}
        C \cdot m^{(1-2/p)_+} \cdot 2^L \cdot \log \log \frac{m^{p'\!/p}}{2^L}
            \stackrel{!}{\leq} n\,.
    \end{equation}
    We will show that there exists a constant $c \in (0,1)$ such that for $m \geq 16n$ the choice
    \begin{equation*}
        L := \max\left\{0,
                \left\lfloor \log_2\left(c \cdot \frac{n}{m^{(1-2/p)_+} \cdot \log \log \frac{m}{n}}\right) \right\rfloor
                \right\}
    \end{equation*}
    ensures~\eqref{eq:L-vs-n} if $L \geq 1$, that is, if
    \begin{equation} \label{eq:n-cond}
        n \geq \frac{2}{c} \cdot m^{(1-2/p)_+} \cdot \log \log \frac{m}{n} \,.
    \end{equation}
    If $n$ violates \eqref{eq:n-cond}, we have $L=0$ and
    we resort to the zero algorithm with cost~$0$ and error~$1$,
    the so-called \emph{initial error}.
    If, however,  \eqref{eq:n-cond} holds, then we have
    \begin{equation} \label{eq:2^L-order}
        \frac{c}{2} \cdot \frac{n}{m^{(1-2/p)_+} \cdot \log \log \frac{m}{n}}
            < 2^L 
            \leq c \cdot \frac{n}{m^{(1-2/p)_+} \cdot \log \log \frac{m}{n}} \,,
    \end{equation}
    which (with $(1-2/p)_+ = 1 - p'\!/p$) leads to
    \begin{equation} \label{eq:Cc-bound}
        C \cdot m^{(1-2/p)_+} \cdot 2^L \cdot \log \log \frac{m^{p'\!/p}}{2^L}
            < C \cdot c \cdot \frac{n}{\log \log \frac{m}{n}}
                    \cdot \log \log \left(\frac{2}{c} \cdot \frac{m}{n} 
                                            \cdot \log \log \frac{m}{n} 
                                    \right) .
    \end{equation}
    For $x := \frac{m}{n} \geq 16$ and $0 < c \leq 2$, observe
    \begin{equation*}
        c \cdot \frac{\log \log \left(\frac{2}{c}\cdot x \cdot \log \log x\right)}{\log \log x}
            \leq c \cdot \frac{\log \log \left(\frac{2}{c} \cdot x^2\right)}{\log \log x}
            \leq c \cdot \frac{\log \log \left(\frac{512}{c}\right)}{\log \log 16}
            \xrightarrow[c \to 0]{} 0 \, .
    \end{equation*}
    This shows that the right-hand side of \eqref{eq:Cc-bound} is smaller or equal~$n$ if we take the constant $c$ sufficiently small.
    Now, combining~\eqref{eq:2^L-order} with the error bound of Theorem~\ref{thm:A_LR}, we obtain
    \begin{equation} \label{eq:e(n)-proof}
        e^\ran(n,\ell_p^m \embed \ell_q^m)
            \leq 3^{1/q} \cdot \left(\frac{2}{c} \cdot \frac{m^{(1-2/p)_+} \cdot \log \log \frac{m}{n}}{n}\right)^{\frac{1}{p'}\left(1 - \frac{p}{q}\right)} .
    \end{equation}
    This was shown assuming~\eqref{eq:n-cond},
    but if this is violated then the right-hand side of~\eqref{eq:e(n)-proof}
    is larger than $3^{1/q}$, which is a trivial upper bound exceeding the initial error.
\end{proof}

\begin{remark}
    We have put some effort into proving upper bounds with $\log \log \frac{m}{n}$ instead of a much simpler bound with the factor $\log \log m$.
    This difference is very subtle: If we take, for instance, $m = m(n) := \lfloor n \log n \rfloor$,
    then
    \begin{equation*}
        \frac{\log \log \frac{m}{n}}{\log \log m}
            \leq \frac{\log \log \log n}{\log \log \lfloor n \log n \rfloor}
            \xrightarrow[n \to \infty]{} 0 \,.
    \end{equation*}
    If, however, we restrict to, say, $n \leq m^\alpha$ for some $\alpha \in (0,1)$, then
    \begin{equation*}
        \log \log m
            \geq \log \log \textstyle\frac{m}{n}
            \geq \log \log m^{1-\alpha}
            = \log \log m - \log \frac{1}{1-\alpha}
            \asymp \log \log m \,.
    \end{equation*}
    The error bound~\eqref{eq:uniformUB} for uniform approximation we found in~\cite{KW24b}
    appears with the factor $\log n + \log \log m$.
    For $n \geq \log m$ this is clearly of order $\log n$.
    For $n \leq \log m$, however, we are in a situation where $\log \log m \asymp \log \log \frac{m}{n}$.
    To summarize,
    \begin{equation*}
        \log n + \log \log m \asymp \log n + \log \log \textstyle\frac{m}{n} \,,
    \end{equation*}
    so the neglecting the reduced size of the buckets in our previous work on uniform approximation
    did \emph{not} lead to worse asymptotic bounds.
\end{remark}

\begin{remark}[Probabilistic error criterion]
    In the context of uniform approximation the analysis of the algorithm started
    in the probabilistic setting of a ``small error with high probability'',
    see~\cite[Thm~3.1]{KW24b}.
    In the current paper, however, we directly go for the expected error, see Theorem~\ref{thm:A_LR}.
    It is well known that by independently repeating an algorithm and combining the results in an appropriate way one can amplify the probability of success
    (compare literature on the median trick, e.g.~\cite{NP09}).
    In~$A_{L,R}$, however, repetition is inherent to the algorithm,
    namely, by choosing $R$ proportional to $\log \delta^{-1}$ we can achieve
    \begin{equation*}
        \sup_{\substack{\vec{x} \in \R^m \\ \|\vec{x}\|_p \leq 1}}
            \P\bigl( \|A_{L,R}(\vec{x}) - \vec{x} \|_q > \eps \bigr) \leq \delta
    \end{equation*}
    in a very cost-effective way.
    Here, $1-\delta$ is the confidence level we aim for.
\end{remark}

\begin{remark}[Lower bounds]
    Heinrich~\cite{H92} proved that for $1 \leq p < q < \infty$ we have
    \begin{equation*}
        e^{\ran}(n,\ell_p^m \embed \ell_q^m)
            \succeq n^{-\left(\frac{1}{p} - \frac{1}{q}\right)}.
    \end{equation*}
    This lower bound shows that the $n$-dependence of the error rate in Theorem~\ref{thm:complexity} is optimal,
    If we consider $m = 16n$, then the upper bound of Theorem~\ref{thm:complexity} matches the lower bound.
    It remains an open problem to show that for $n \ll m$
    the $m$-dependence of our adaptive upper bounds is optimal as well.
\end{remark}

\subsection{Gap between adaptive and non-adaptive methods}

We state a result in the style of~\eqref{eq:gap1oo} and \eqref{eq:gap12}
(shown in \cite{KNW24,KW24b}) 
which concerns the gap between adaptive and non-adaptive methods
but for a wider range of parameters.
Here we use the upper bounds for adaptive methods from this paper, see Theorem~\ref{thm:complexity}.
We contrast this with lower bounds for non-adaptive methods from~\cite{KNW24}, see~\eqref{eq:nonadaLB}.
For the sake of completeness,
we include the case $q=\infty$ with the upper bounds from~\cite[Thm~3.3]{KW24b}.
We omit the proof as it follows exactly the lines of the proofs of the previous results~\eqref{eq:gap1oo} and \eqref{eq:gap12}.

\begin{theorem} \label{thm:adagap3}
    Let $n \in \N$ and $m = m(n) := \bigl\lceil C \, e^{an^2} \bigr\rceil$ with the contants $C,a > 0$ from~\cite[Thm~2.7]{KNW24}. Then, for $1 \leq p \leq 2$ and $p < q \leq \infty$ we have
    \begin{equation*}
        \frac{e^{\ran}(n,\ell_p^m \embed \ell_q^m)
            }{e^{\ran,\nonada}(n,\ell_p^m \embed \ell_q^m)}
                \preceq \left(\frac{\log n}{n}\right)^{\frac{1}{p} - \frac{1}{q}}.
    \end{equation*}
\end{theorem}

It is left open to show that this gap is as big as it can get for the specific combinations of summability indices $p$ and $q$, except for the special cases with $(p,q) \in \{(1,2), (1,\infty),(2,\infty)\}$ that have already been covered in~\cite{KNW24,KW24b}.
Let us point out that for $p > 2$ no such proven gap is known to us,
yet our upper bounds suggest the existence of a small logarithmic gap.
Results in that direction will require new lower bounds for non-adaptive approximation in the case of $p>2$.

\section{Non-adaptive methods}
\label{sec:nonada}

In Sections~\ref{sec:linsketch} and \ref{sec:countsketch} we will consider two basic algorithms: $\linsketch$ which proves useful in the case $p\geq 2$, and $\countsketch$ which is appropriate for \mbox{$1 \leq p < 2$}.
We need to modify these algorithms by denoising their outputs to achieve optimal rates if $q<\infty$, see Section~\ref{sec:denoise}. The results will be summarized in Section~\ref{sec:nonadasummary}.

\subsection{A linear Monte Carlo method}
\label{sec:linsketch}

A fairly simple randomized non-adaptive method for approximating the embedding
$\ell_2^m \embed \ell_{\infty}^m$ was described by Math\'e~\cite{Ma91}.
The information mapping can be represented by a matrix $N \in \R^{n \times m}$ with independent standard Gaussian entries,
and for $\vec{x} \in \R^m$ we define the output
\begin{equation*}
    \linsketch_n(\vec{x}) := \frac{1}{n}N^\top N \vec{x} \,.
\end{equation*}
This method is linear, for $m \geq 2$ we have
\begin{equation*}
    e(\linsketch_n, \ell_2^m \embed \ell_{\infty}^m)
        \leq 2\,\sqrt{\frac{2 \log m}{n}} \,.
\end{equation*}
More generally, for $p \geq 2$, exploiting $\|\cdot\|_2 \leq m^{\frac{1}{2} - \frac{1}{p}} \|\cdot\|_p$ in $\R^m$, we find
\begin{equation}\label{eq:linsketcherror}
    e(\linsketch_n, \ell_p \embed \ell_{\infty}^m)
        \leq 2\,\sqrt{\frac{2 \, m^{1 - 2/p} \log m}{n}} \,.
\end{equation}
This bound is improving over the initial error~$1$ only for
$n \geq 8 \, m^{1-2/p} \log m$.
$\linsketch$ can also be analysed for the $\ell_q$-error with $2<q<\infty$ giving
\begin{equation}\label{eq:LinSketch-ell_q}
    e(\linsketch_n, \ell_2^m \embed \ell_q^m)
        \asymp \frac{m^{1/q}}{\sqrt{n}}
\end{equation}
with $q$-dependent implicit constants.
The upper bound relies on \cite[Prop~3.1]{Ku17} and
results on the expected $\ell_q$-norm of Gaussian vectors.
The lower bound is obtained by considering
\mbox{$\vec{x} = (1,0, \ldots, 0)$}.
As it turns out, via a non-linear modification ("denoising") of this method we can get rid of the polynomial $m$-dependence for $q < \infty$.
We will describe this approach later after introducing $\countsketch$,
another non-adaptive method that is already non-linear from the beginning.

\subsection{Count sketch}
\label{sec:countsketch}

For the problem $\ell_p^m \embed \ell_\infty^m$ with $p < 2$,
we use $\countsketch$, a non-linear randomized approximation method
first developed in~\cite{CCF04}.
This is method is closely related to the bucket selection scheme from~\cite[Sec~2.2]{KW24b}, see also~\cite[Lem~54]{LNW17}, and has been mentioned in~\cite[Rem~2.6]{KW24b}.

For algorithmic parameters $R,G \in \N$ where $R$ is odd,
we generate independent hash vectors~$\vec{H}^{(1)},\ldots,\vec{H}^{(R)} \iid \Uniform[G]^m$
and draw random signs
\mbox{$\sigma_{ri} \iid \Uniform\{\pm1\}$}, \mbox{$r \in [R]$}, \mbox{$i \in [m]$}.
The algorithm takes $R \cdot G$ Rademacher measurements:
\begin{equation*}
    Y_{r,g} = L_{r,g}(\vec{x}) := \sum_{i \in [m] \colon H_i^{(r)} = g} \sigma_{ri} \cdot x_i
    \,,
    \qquad r \in [R],\, g \in [G].
\end{equation*}
That way we execute
$R$ repetitions of a grouped measurement where for each repetition the coordinates are 
randomly sorted into $G$ groups on which we perform individual measurements.
Hence, each coordinate $i \in [m]$ exerts influence on exactly $R$ measurements $Y_{r,g}$.
We use these to define the following variables:
\begin{equation*}
    \hat{Y}_{r,i} := \sigma_{ri} Y_{r,H_i^{(r)}} \,, \qquad r \in [R].
\end{equation*}
The output~$\vec{Z} := \countsketch_{R,G}$ of the method is defined using the median for each component in the following way:
\begin{equation*}
    Z_i := \median \left\{\hat{Y}_{r,i} \;\middle|\; r \in [R] \right\}.
\end{equation*}
We provide a precise
error analysis for this method in our setting.

\begin{proposition}
    Let $1 \leq p \leq 2$ and $\vec{x} \in \R^m$ with $\|\vec{x}\|_p \leq 1$.
    For $L \in \N$ let $G := 2^{4+L}$ and pick~$R$ as the smallest odd number such that
    \begin{equation*}
        R \geq \max\{5, 2+ 3\log_2 m\} \,.
    \end{equation*}
    Then we obtain the error bound
    \begin{equation*}
        e(\countsketch_{R,G}, \ell_p^m \embed \ell_\infty^m)
            \leq 4 \cdot 2^{-\frac{L}{p}}
    \end{equation*}
    while the cardinality of the method is
    \begin{equation*}
        \card(\countsketch_{R,G}) = R \cdot G \asymp 2^{L} \cdot \log m \,.
    \end{equation*}
\end{proposition}
\begin{proof}
    For $k \in [m]$ define $\eps_k := k^{-1/p}$.
    Then, for any given vector~$\vec{x} \in \R^m$ with $\|\vec{x}\|_p \leq 1$,
    we find that at most $k$~entries can have an absolute value larger than~$\eps_k$.
    Let $Q_k \subseteq [m]$ be a (not necessarily unique) set of $k$~coordinates with the largest absolute values, i.e.~$\#Q_k = k$ and
    \begin{equation*}
        \min_{i \in Q_k} |x_i| \geq \max_{i \in [m]\setminus Q_k} |x_i| \,.
    \end{equation*}
    Consider a fixed coordinate $i \in [m]$ and for each round $r \in [R]$ define the set of companion coordinates that are measured together with the $i$-th coordinate in that round:
    \begin{equation*}
        C_i^{(r)} := \left\{ j \in [m] \setminus \{i\} \colon H_i^{(r)} = H_j^{(r)} \right\} .
    \end{equation*}
    By a union bound we estimate the probability that large coordinates are among the companion coordinates of~$i$:
    \begin{equation} \label{eq:hitk}
        \P\left( Q_k \cap C_i^{(r)} \neq \emptyset\right) \leq \frac{k}{G} \,.
    \end{equation}
    Obviously, $\hat{Y}_{r,i}$ is unbiased:
    $\expect \hat{Y}_{r,i} = x_i$.
    We compute the variance of~$\hat{Y}_{r,i}$ conditioned on the event that no large entries occur among the companion coordinates of~$i$.
    We do this by analysing $\hat{Y}_{r,i}$ as a sum of independent centred random variables
    \mbox{$\ind_{[{H_j}^{(r)} = H_i^{(r)}]} \cdot \sigma_{ri}\sigma_{rj} \cdot x_j$}
    for \mbox{$j \in [m] \setminus (Q_k \cup \{i\})$}:
    \begin{equation} \label{eq:condvar}
        \expect\left[\left(\hat{Y}_{r,i} - x_i\right)^2 \;\middle|\; Q_k \cap C_i^{(r)} = \emptyset \right]
            \leq \frac{1}{G}  \cdot \|\vec{x}_{[m] \setminus Q_k}\|_2^2
            \leq \frac{k^{1-2/p}}{G} \,.
    \end{equation}
    (In the last inequality of \eqref{eq:condvar} we used
    a well-known result on best $k$-term approximation which states
    $\|\vec{x}_{[m] \setminus Q_k}\|_q \leq k^{-\left(\frac{1}{p}-\frac{1}{q}\right)} \|\vec{x}\|_p$ for $p < q$,
    see e.g.~\cite[eq~(2.6)]{CDD09},
    here with $q=2$ and for $\|\vec{x}\|_p \leq 1$.)
    With Chebyshev's inequality applied to~\eqref{eq:condvar}, and together with~\eqref{eq:hitk}, we find
    \begin{align*}
       \P\left( |\hat{Y}_{r,i} - x_i| > \eps_k\right) 
      & =  \P\left(  |\hat{Y}_{r,i} - x_i| > \eps_k  \; \middle|\; Q_k \cap C_i^{(r)} = \emptyset  \right) \cdot
        \P\left(Q_k \cap C_i^{(r)} = \emptyset \right) 
      \\
      & \quad +  \P\left(  |\hat{Y}_{r,i} - x_i| > \eps_k  \; \middle|\; Q_k \cap C_i^{(r)} \neq \emptyset  \right) \cdot
        \P\left(Q_k \cap C_i^{(r)} \neq \emptyset \right)   
      \\
      & \leq \left(\frac{k^{1 - 2/p}}{G} \cdot \eps_k^{-2}\right) \cdot 1
        + 1 \cdot \frac{k}{G}
       = \frac{2k}{G} =: \alpha \,.
    \end{align*}
    Taking the median of several independent estimates amplifies the probability of success.
    In detail, from~\cite[eq~(2.6)]{NP09} 
    we conclude that for $0 < \alpha < \frac{1}{4}$, thus for $G > 8k$,
    with $Z_i$ being the median of~$R$ independent copies of $\hat{Y}_{r,i}$, we have
    \begin{equation*}
        \P\left( |Z_i - x_i| > \eps_k\right) 
            \leq \frac{1}{2} \left(4\alpha\right)^{R/2}
            = \frac{1}{2} \left(\frac{8k}{G}\right)^{R/2} .
    \end{equation*}
    A union bound over all coordinates
    gives the following result on the uncertainty for uniform approximation:
    \begin{equation*}
        \P\left( \|\vec{Z} - \vec{x}\|_\infty > \eps_k\right) 
            \leq \frac{m}{2} \left(\frac{8k}{G}\right)^{R/2}.
    \end{equation*}
    On the other hand, we know an absolute bound for the error of Rademacher measurements,
    namely $\|\vec{Z} - \vec{x}\|_\infty \leq \|\vec{x}\|_1 \leq m^{1-1/p}$.
    Altogether, for $G = 2^{4+L}$ and with values $k=2^l$ for $l=0,\ldots,L$,
    the expected error can be estimated as follows:
    \begin{align*}
        \expect  &\|\vec{Z} - \vec{x}\|_\infty
        \\
            &\leq m^{1-\frac{1}{p}} \cdot \P\left(\|\vec{Z} - \vec{x}\|_\infty > \eps_1\right)
                + \sum_{l=1}^{L} \eps_{2^{l-1}} \cdot \P\left(\|\vec{Z} - \vec{x}\|_\infty > \eps_{2^l}\right)
                + \eps_{2^L} \\
            &\leq m^{2-\frac{1}{p}} \cdot 2^{-\frac{(L+1)R}{2} - 1}
                + \sum_{l=1}^L m \cdot 2^{-\frac{l-1}{p} - \frac{(L-l+1)R}{2} - 1}
                + 2^{-\frac{L}{p}} \\
            &= m^{2-\frac{1}{p}} \cdot 2^{-\frac{(L+1)R}{2} - 1}
                + m \cdot 2^{- \frac{L-1}{p} - \frac{R}{2} - 1}
                    \cdot \sum_{l=1}^{L} \left(2^{\frac{1}{p} - \frac{R}{2}}\right)^{L-l}
                + 2^{-\frac{L}{p}} .
    \end{align*}
    If we take $R > 4$, then the sum is bounded by~$2$.
    Further, $R \geq 2 + 3\log_2m$ ensures that the first term is bounded by $2^{-L/p}$,
    as well as the pre-factor of the sum, leading to
    $\expect \|\vec{Z} - \vec{x}\|_\infty \leq 4 \cdot 2^{-L/p}$.
\end{proof}

\begin{corollary}\label{cor:countsketcherror}
    Let $1 \leq p \leq 2$ and $m \in \N$, $m \geq 2$.
    Given $\eps \in (0,1)$ we find the following asymptotic cardinality bound for non-adaptive Monte Carlo methods:
    \begin{equation*}
        n^{\ran,\nonada}(\eps,\ell_p^m \embed \ell_\infty^m)
            \preceq \eps^{-p} \cdot \log m \,.
    \end{equation*}
    Conversely, for $n < m$ we find the following error rate:
    \begin{equation*}
        e^{\ran,\nonada}(n,\ell_p^m \embed \ell_\infty^m)
            \preceq \left(\frac{\log m}{n} \right)^{\frac{1}{p}} .
    \end{equation*}
\end{corollary}

Note that for~$p=2$ the rate obtained with $\countsketch$ is identical to the rate of $\linsketch$, but $\linsketch$ is the simpler algorithm with smaller constant in the error estimate.

\subsection{Denoising the output}
\label{sec:denoise}

The problem of the above uniform approximation algorithms is
that the output is quite noisy (with large $\ell_2$-norm),
potentially leading to an unnecessarily large error
if measured in the $\ell_q$-norm with $q< \infty$.
We solve this problem by keeping only $k$ entries of the reconstruction with the largest absolute values and putting all other coordinates to zero.
Note that by this we introduce a bias to the method,
see~\cite{Ma96} for a study of unbiased approximation methods.

\begin{lemma}\label{lem:denoising}
    Let $1 \leq p < q \leq \infty$ and let $A$ be an algorithm with
    \begin{equation*}
       e(A,\ell_p^m \embed \ell_\infty^m) \leq C \cdot \eps
    \end{equation*}
    where $C \geq 1$ and $\eps > 0$.
    For $\vec{x} \in \R^m$ denote the corresponding output $\vec{z} := A(\vec{x})$.
    Define $k := \min\{\lfloor \eps^{-p} \rfloor, m\}$
    and choose a $k$-element set $I_k = I_k(\vec{z}) \subseteq [m]$ satisfying
    \begin{equation*}
        \min_{i \in I_k} |z_i| \geq \max_{i \in [m]\setminus |I_k} |z_i|\,.
    \end{equation*}
    Based on this we define the denoised algorithm $D$ with output $\vec{w} = D(\vec{x})$ where
    \begin{equation*}
        w_i = \begin{cases}
            z_i &\text{if } i \in I_k(\vec{z}),\\
            0 &\text{else.}
        \end{cases}
    \end{equation*}
    Then
    \begin{equation*}
        \textstyle \frac{1}{3} \, \eps^{1 - \frac{p}{q}}
            \leq  e(D,\ell_p^m \embed \ell_q^m)
            \leq (1 + 5C) \cdot \eps^{1-\frac{p}{q}} \,
    \end{equation*}
    where the lower bound holds under the additional assumption $2k+1 \leq m$.
    Moreover,
    \begin{equation*}
        \card D = \card A \,.
    \end{equation*}
\end{lemma}
\begin{proof}
    The fact that 
        $\card D = \card A$ 
    is obvious as 
    $D$ is only adding a post-processing step without any extra measurements.
    
    We start by showing the upper error bound.
    We obtain the $\omega$-wise $\ell_q$-error by interpolating between $\ell_p$- and $\ell_\infty$-error, compare~\eqref{eq:Interpolation}:
    \begin{equation} \label{eq:Interpolation2}
       \| D^\omega(\vec{x}) - \vec{x}\|_q
       \leq \| D^\omega(\vec{x}) - \vec{x} \|_p^{p/q} 
                \cdot \|  D^\omega(\vec{x}) - \vec{x} \|_{\infty}^{1 - p/q} .
    \end{equation}
    Given an $\vec{x} \in \R^m$
    denote the $\ell_\infty$-error at the current instance by
    \begin{equation*}
        \mathcal{E}_{\vec{x}}
            = \mathcal{E}^{\omega}_{\vec{x}}
            := \| A^\omega(\vec{x}) - \vec{x}\|_{\infty}\,.
    \end{equation*}
    Thus $\mathcal{E}_{\vec{x}}$ is a random variable
    with $\expect \mathcal{E}_{\vec{x}} \leq C \cdot \eps$.
    For~$\|\vec{x}\|_p \leq 1$, by the choice of~$k$, there are at most~$k$ entries of~$\vec{x}$ with $|x_i| \geq \eps$.
    Hence, there are at most~$k$ entries of~$\vec{z}^\omega = A^\omega(\vec{x})$ with $|z_i^\omega| \geq \eps+\mathcal{E}_{\vec{x}}^\omega$.
    Therefore, if for an entry of $\vec{w}^\omega = D^\omega(\vec{x})$ we have $w_i^\omega = 0$, then $|z_i^\omega| < \eps+\mathcal{E}_{\vec{x}}^\omega$ and
    $|w_i^\omega - x_i| \leq |w_i^\omega - z_i^\omega| + |z_i^\omega - x_i| < \eps + 2\mathcal{E}_{\vec{x}}^\omega$.
    If, however, $w_i^\omega \neq 0$, then $w_i^\omega = z_i^\omega$ and $|w_i^\omega - x_i| = |z_i^\omega - x_i| \leq \mathcal{E}_{\vec{x}}^\omega$.
    This shows
    \begin{equation*}
        \|\vec{w}^\omega - \vec{x}\|_\infty
        = \|D^\omega(\vec{x}) - \vec{x}\|_\infty
        \leq \eps + 2\mathcal{E}_{\vec{x}}^\omega.
    \end{equation*}
    (In particular, $\expect\|D(\vec{x}) - \vec{x}\|_\infty \leq (1+2C) \cdot \eps$, so denoising does not deteriorate the error guarantees for uniform approximation by much.)
    Furthermore, 
    \begin{equation}\label{eq:2NormBound}
      \| D^{\omega}(\vec{x}) - \vec{x} \|_p^p 
      = \sum_{j \in I_k} |z_j^{\omega} - x_j|^p +
      \| \vec{x}_{I_k^c} \|_p^p 
      \leq  k \cdot (\mathcal{E}^{\omega}_{\vec{x}})^p + 1 \,.
    \end{equation}
    Substituting \eqref{eq:2NormBound} into \eqref{eq:Interpolation2},
    taking the expectation, and keeping in mind that $k \leq \eps^{-p}$,
    we find
    \begin{align*}
        \expect \| D(\vec{x}) - \vec{x}\|_q
            & \leq \expect \left[ (k \mathcal{E}_{\vec{x}}^p + 1)^{1/q}
                                \cdot (\eps+2\mathcal{E}_{\vec{x}})^{1 - p/q}
                            \right] \\
            & \leq \expect \left[ (k^{1/q}\mathcal{E}_{\vec{x}}^{p/q} + 1)
                    \cdot (\eps^{1-p/q} + (2\mathcal{E}_{\vec{x}})^{1 - p/q})  
                    \right] \\
            & = k^{1/q} \eps^{1-p/q} \cdot\expect [\mathcal{E}_{\vec{x}}^{p/q}]
                + 2 k^{1/q} \cdot\expect \mathcal{E}_{\vec{x}}
                + \eps^{1-p/q}
                + \expect [(2\mathcal{E}_{\vec{x}})^{1 - p/q}] \\
            &\leq \eps^{1-2p/q} \cdot \left(\expect \mathcal{E}_{\vec{x}}\right)^{p/q}
                + 2\eps^{-p/q} \cdot \expect{\mathcal{E}_{\vec{x}}}
                + \eps^{1-p/q}
                + \left(2 \expect \mathcal{E}_{\vec{x}}\right)^{1-p/q} \\
            &\leq \left(C^{p/q} + 2C + 1 + (2C)^{1-p/q} \right) \eps^{1-p/q} \\
            &\leq (1+5C) \cdot \eps^{1-p/q} .
    \end{align*}
    Since this holds for all~$\vec{x}$ with $\|\vec{x}\|_p \leq 1$,
    we proved the upper bound.

    We now show the lower bound, starting
    with the case $q = \infty$.
    Taking as an input vector $\vec{x}$ with some $(k+1)$ 
    entries set to
    $(\lfloor  \eps ^{- p}\rfloor + 1)^{- 1/p}$
    and all the other
    entries set to $0$,
    we have $\| \vec{x} \|_p = 1$ and for $\vec{w} = D(\vec{x})$ we get
    \begin{equation*}
       \| \vec{w} - \vec{x} \|_{\infty} 
            \geq (\lfloor  \eps ^{- p}\rfloor + 1)^{- 1/p}
            \geq 2^{-1/p}\, \eps \,.
    \end{equation*}
    To prove the lower bound in the case $q < \infty$ consider an input vector $\vec{x}$ having some $2k+1$ entries set to
    $(2k+1)^{-1/p} = (2\lfloor \eps^{-p} \rfloor + 1)^{-1/p}$
    and all the other entries set to~$0$.
    Once again $\| \vec{x} \|_p = 1$ and
    \begin{align*}
        \| \vec{w} - \vec{x} \|_q
            &\geq \left((k+1) \cdot  (2k+1)^{- q/p}\right)^{1/q} \\
            &= \left((\lfloor\eps^{-p}\rfloor + 1)
                            \cdot (2\lfloor \eps^{-p} \rfloor + 1)^{-q/p}
                \right)^{1/q} \\
            &\geq 
                3^{-1/p}\,\eps^{1 - p/q} .
    \end{align*}
    This finishes the proof
\end{proof}

We apply the above lemma to the non-adaptive algorithms we introduced before.
For $1 \leq p < 2$ we take $\countsketch$ with the $\ell_\infty$-error rates from Corollary~\ref{cor:countsketcherror},
and, combined with the choice~$k := \left\lfloor \frac{n}{\log m} \right\rfloor$
for the denoised algorithm, we obtain
\begin{equation} \label{eq:nonada,p<=2}
    e^{\ran,\nonada}(n,\ell_p^m \embed \ell_q^m) 
        \preceq \left(\frac{\log m}{n}\right)^{\frac{1}{p} - \frac{1}{q}} .
\end{equation}

For~$2 \leq p < \infty$, employing $\linsketch$ with the error rates for $\ell_p^m \embed \ell_\infty^m$ from~\eqref{eq:linsketcherror}, 
and with $k := \left\lfloor \left(\frac{n}{m^{1 - 2/p} \log m}\right)^{p/2}\right\rfloor$, 
we find
\begin{equation}
    e^{\ran,\nonada}(n,\ell_p^m \embed \ell_q^m)
        \preceq
            \left( \frac{m^{1-2/p} \cdot \log m}{n}\right)^{\frac{1}{2}\left(1 - \frac{p}{q}\right)}.
        \label{eq:2<p<q<oo,nonada}
\end{equation}
Note that for $n < m^{1-2/p} \log m$ we have~$k=0$, that is, we fall back on the zero algorithm
which is essentially the best we can do in this setting.
Alternatively we could use the relation
\begin{equation} \label{eq:pq_via_2q,klon}
    e^{\ran,\nonada}(n,\ell_p^m \embed \ell_q^m)
        \leq m^{\frac{1}{2}-\frac{1}{p}} \cdot e^{\ran,\nonada}(n,\ell_2^m \embed \ell_q^m) \,,
\end{equation}
exploiting $\|\vec{x}\|_2 \leq m^{\frac{1}{2} - \frac{1}{p}} \|\vec{x}\|_p$ for $\vec{x} \in \R^m$.
If we combine~\eqref{eq:pq_via_2q,klon} with
\begin{equation*}
    e^{\ran,\nonada}(n,\ell_2^m \embed \ell_q^m) \preceq \left(\frac{\log m}{n}\right)^{\frac{1}{2} - \frac{1}{q}},
\end{equation*}
the result will be consistently worse than~\eqref{eq:2<p<q<oo,nonada}.
If, however, we use the direct $\ell_q$-error analysis of undenoised $\linsketch$~\eqref{eq:LinSketch-ell_q},
together with~\eqref{eq:pq_via_2q,klon} we find
\begin{equation} \label{eq:2<p<q<oo,large_n}
    e^{\ran,\nonada}(n,\ell_p^m \embed \ell_q^m)
        \leq m^{\frac{1}{2} + \frac{1}{q} - \frac{1}{p}} \cdot \frac{1}{\sqrt{n}}\,.
\end{equation}
This bound is better than~\eqref{eq:2<p<q<oo,nonada} only for the narrow window
\begin{equation*}
    n \succ \frac{m}{(\log m)^{\frac{q}{p} - 1}} \,,
\end{equation*}
so for $n\ll m$ in the sense of~$n \leq m^\alpha$ for some $\alpha \in (0,1)$,
the bound~\eqref{eq:2<p<q<oo,large_n} does not play any role.

\subsection{Summary on non-adaptive results}
\label{sec:nonadasummary}

We summarize our findings for non-adaptive methods in the following theorem.
Here, for the sake of simplicity,
we focus on the best results results we know for the regime~$n \ll m$ in the sense that $n \leq m^\alpha$ for some~$\alpha \in (0,1)$.
For larger~$n$ close to~$m$, in particular $m=2n$, we find better results in~\cite[Sec~4]{H92}.

\begin{theorem} \label{thm:nonada-summary}
    Let $1 \leq p < q \leq \infty$, $m,n \in \N$, $m \geq 2$.
    For non-adaptive randomized methods we have the following asymptotic upper bounds where the implicit constant may depend on $p$ and $q$:
    \begin{equation*}
        e^{\ran,\nonada}(n,\ell_p^m \embed \ell_q^m)
            \preceq
                \begin{cases}
                    \displaystyle
                    \min\left\{1,\, \left(\frac{\log m}{n}\right)^{\frac{1}{p} - \frac{1}{q}}
                        \right\}
                        &\text{for $1 \leq p \leq 2$,} \vspace{2pt}\\
                    \displaystyle
                    \min\left\{1,\,
                                \left( \frac{m^{1-2/p} \cdot \log m}{n}\right)^{\frac{1}{2}\left(1 - \frac{p}{q}\right)}
                                \right\}
                        &\text{for $p > 2$.}
                \end{cases}
    \end{equation*}
    Conversely, for~$\eps \in (0,1)$ we find
    \begin{equation*}
        n^{\ran,\nonada}(\eps,\ell_p^m \embed \ell_q^m)
            \preceq \begin{cases}
                \eps^{-\left. 1 \middle/\left(\frac{1}{p} - \frac{1}{q}\right)\right.}
                    \cdot \log m
                    &\text{for } 1 \leq p \leq 2,\\
                \eps^{-\left. \frac{2}{p}\middle/\left(\frac{1}{p} - \frac{1}{q}\right)\right.}
                        \cdot m^{1 - \frac{2}{p}} \cdot  \log m
                    &\text{for } p > 2.
            \end{cases}
    \end{equation*}
\end{theorem}

\begin{remark}[Deterministic methods for $q \leq 2$]
    For \mbox{$1 \leq p < q \leq 2$} it is well known that there exist non-linear deterministic (non-adaptive) methods achieving
    \begin{equation} \label{eq:det,q<=2}
        e^{\deter}(n, \ell_p^m \embed \ell_q^m)
            \asymp \min\left\{1,\; 
                m^{1 - \frac{1}{p}} \cdot \left(\frac{\log \frac{m}{n}}{n}\right)^{1 - \frac{1}{q}}
                        \right\},
    \end{equation}
    which is due to~\cite{Ka74,GG84} (originally formulated in terms of the dual quantity of Kolmogorov numbers).
    It was unknown so far if randomized algorithms can achieve better rates.
    For $p=1$, the deterministic rate~\eqref{eq:det,q<=2} is already slightly better than denoised $\countsketch$ if $n$ is close to $m$, see~\eqref{eq:nonada,p<=2}.
    However, if we restrict to $n \ll m$
    in the sense of $n \leq m^\alpha$ for some $\alpha \in (0,1)$,
    both error rates are of the same order
    because $\log m \asymp \log \frac{m}{n}$ in that regime.
    Significantly, it seems that for $1=p < q \leq 2$ randomization does not help as long as we restrict to non-adaptive methods.
    For $p > 1$, in contrast, it turns out that already non-adaptive randomized algorithms can perform significantly better than deterministic methods if~$n\ll m$.
\end{remark}

\appendix

\section{Some results on asymptotic relations}
\label{app:asymp}

Asymptotic relations that involve logarithms are often quite surprising.
For instance, if we have weakly asymptotically equivalent functions $f(m) \asymp g(m)$ with $f(m),g(m) \to \infty$ for $m \to \infty$,
then their logarithms are strongly asymptotically equivalent, $\log f(m) \simeq \log g(m)$.
This is what we used in~\eqref{eq:k*} to simplify the order of~$k^\ast(m/D)$.
Asymptotic results with iterated logarithms in particular pose the problem that we need to be aware of the domain on which the logarithm takes positive values,
which is why the definition~\eqref{eq:k*} of $k^\ast$ works for general $m > D$
but the weak asymptotic cost bound~\eqref{eq:n*} for $\spot$, namely $\log \log \frac{m}{D}$, is stated for $m \geq 16D$.
Later on we want to state error bounds with a factor $\log \log \frac{m}{n}$.
Since $D$ is much smaller than $n$, the natural restriction~$m \geq 16n$ is stronger.
On other occasions, however, after changing the argument of the double logarithm we want to relax the restrictions for the validity of the asymptotic relation to keep statements as simple as possible.
The following abstract result shows how the domain of an asymptotic estimate can be extended under certain circumstances, we subsequently provide an example from this paper.

\begin{lemma}\label{lem:extendasymp}
    Let $h_1,h_2 \colon \N \to \N$ with $h_1(n) \geq h_2(n)$, let $c_1,c_2 > 0$, and consider functions
    \begin{align*}
        e \colon &\N^2 \to [0,\infty)\,, \\
        f_1 \colon & \left\{(n,m) \in \N^2 \mid m \geq h_1(n) \right\} \to [0,\infty)\,, \\
        f_2 \colon & \left\{(n,m) \in \N^2 \mid m \geq h_2(n) \right\} \to (0,\infty)\,.
    \end{align*}
    Assume that $e(n,m)$, $f(n,m)$, and $g(n,m)$ are all
    monotonically increasing in~$m$.
    Further assume
    \begin{equation*}
        b := \sup_{n \in \N} \frac{f_2(n, h_1(n))}{f_2(n,h_2(n))} < \infty \,.
    \end{equation*}
    Then the following implication holds:
    \begin{align*}
        e(n,m) \preceq f_1(n,m) &\preceq f_2(n,m) \qquad\text{for } m \geq h_1(n) \\
        \Longrightarrow\qquad
        e(n,m) &\preceq f_2(n,m) \qquad\text{for } m \geq h_2(n) \,.
    \end{align*}
\end{lemma}
\begin{proof}
    The premise states that there exists a constant $C > 0$ such that
    \begin{equation*}
        e(n,m) \leq C \cdot f_1(n,m) \quad\text{and}\quad f_1(n,m) \leq C \cdot f_2(n,m) \quad\text{for } m \geq h_1(n) \,,
    \end{equation*}
    in particular $e(n,m) \leq C^2 \cdot f_2(n,m)$ for $m \geq h_1(n)$.
    Monotonicity in~$m$ and the definition of~$b$ imply
    that for $h_1(n) \geq m \geq h_2(n)$ we have
    \begin{equation*}
        e(n,m)
            \leq e(n,h_1(n))
            \leq C^2 \cdot f_2(n,h_1(n))
            \leq C^2 \, b \cdot f_2(n,h_2(n))
            \leq C^2 \, b \cdot f_2(n,m) \,.
    \end{equation*}
    Hence, with $b \geq 1$, we even find
    \begin{equation*}
        e(n,m) \leq C^2 \, b \cdot f_2(n,m)
        \qquad\text{for }m \geq h_2(n) \,,
    \end{equation*}
    which proves the assertion.
\end{proof}

We exemplify this result by a particular application from this paper:

\begin{example} \label{ex:e(L,m)}
    Consider the cost function
    $e(L,m) := \card A_{L,R}$
    from Theorem~\ref{thm:A_LR},
    with asymptotic bound $f_1(L,m) := D^{(L)} \cdot \log \log \frac{m}{D^{(L)}}$
    for $m \geq h_1(L) := 16 \, D^{(L)}$ where $D^{(L)} := \lceil C_p \cdot 2^L \rceil$ with $C_p > 1$ for $1 \leq p \leq 2$.
    Further, we know that \mbox{$f_1(L,m) \leq (C_p+1) \cdot f_2(L,m)$} where $f_2(L,m) := 2^L \cdot \log \log \frac{m}{2^L}$,
    but $f_2(L,m)$ can be regarded on the domain \mbox{$m \geq h_2(L) := 16 \cdot 2^L$}.
    Now, with $\log \log 16 > 1 > 0$, we find
    \begin{equation*}
        \frac{f_2(L, h_1(L))}{f_2(L,h_2(L))}
            = \frac{\log \log \frac{16 \lceil C_p \cdot 2^L\rceil}{2^L}}{\log \log 16}
            \leq \frac{\log \log (16 (C_p + 1))}{\log \log 16} < \infty \,.
    \end{equation*}
    Thus, by Lemma~\ref{lem:extendasymp} we conclude
    $e(L,m) \preceq 2^L \cdot \log \log \frac{m}{2^L}$ for $m \geq 16 \cdot 2^L$.
\end{example}

\bibliographystyle{amsplain}

\bibliography{lit}

\end{document}